\documentclass[a4paper,11pt]{article}

\usepackage{graphicx}
\usepackage{tikz}

\usepackage{subcaption}

 \usetikzlibrary{external} 

\usepackage{mathrsfs}
\usepackage{booktabs}
\usepackage{graphicx}

\usepackage{amsmath, amssymb}
\usepackage{amsthm}
\usepackage{enumerate}
\usepackage{mathtools}
\usepackage{tikz-cd}
\usepackage{comment}
\usepackage{color}
\usepackage{hyperref}
\usepackage{soul}

\usepackage{stmaryrd}
\usepackage{soul}

\newcommand{\Spec}{\operatorname{Spec}}
\newcommand{\Aut}{\operatorname{Aut}}
\newcommand{\AAut}{\operatorname{AAut}}

\newcommand{\CAT}{\operatorname{CAT}}

\newcommand{\p}{\operatorname{\mathbb{P}}}

\newcommand{\Z}{\operatorname{\mathbb{Z}}}

\newcommand{\A}{\operatorname{\mathbb{A}}}

\newcommand{\R}{\operatorname{\mathbb{R}}}
\newcommand{\F}{\operatorname{{\mathbb{F}}}}
\newcommand{\kk}{{k}}

\newcommand{\id}{\operatorname{id}}

\newcommand{\Bir}{\operatorname{Bir}}

\newcommand{\PGL}{\operatorname{PGL}}

\newcommand{\T}{\mathcal{T}}
\newcommand{\Sym}{\operatorname{Sym}}

\newcommand{\B}{\mathcal{B}}

\newcommand{\Homeo}{\operatorname{Homeo}}

\newcommand{\FF}{\mathcal{F}}

\DeclareMathOperator{\BBir}{BBir}

\usepackage[shortlabels]{enumitem}
\setlist[enumerate]{label=\rm{(\arabic*)}}
\setlist[enumerate,2]{label=\rm({\it\roman*})}
\setlist[itemize]{label=\raisebox{0.25ex}{\tiny$\bullet$}}

\usepackage{graphicx}
\usepackage{tikz}

\usepackage[T1]{fontenc}
\usepackage{lmodern}

\usepackage{dsfont}

\usepackage{amsmath}
\usepackage{amsthm}
\usepackage{amssymb}
\usepackage{mathrsfs}
\usepackage{graphicx}
\usepackage[all]{xy}

\usepackage{xcolor}

\usepackage{makeidx}

\usepackage{stmaryrd}
\usepackage{caption}

\usepackage{abstract} 

\newtheorem{theorem}{Theorem}[section]

\newtheorem{claim}[theorem]{Claim}
\newtheorem{fact}[theorem]{Fact}

\newtheorem{lemma}[theorem]{Lemma}
\newtheorem{proposition}[theorem]{Proposition}

\theoremstyle{definition}
\newtheorem{definition}[theorem]{Definition}

\newtheorem{remark}[theorem]{Remark}

\def\rquotient#1#2{%
	\makeatletter
	\raise.3ex\hbox{$#1$}/\lower.3ex\hbox{$#2$}%
	\makeatother
}	

\makeatletter
\newcommand{\subjclass}[2][2010]{%
	\let\@oldtitle\@title%
	\gdef\@title{\@oldtitle\footnotetext{#1 \emph{Mathematics subject classification.} #2}}%
}
\newcommand{\keywords}[1]{%
	\let\@@oldtitle\@title%
	\gdef\@title{\@@oldtitle\footnotetext{\emph{Key words and phrases.} #1.}}%
}
\makeatother

\newcommand{\Address}{{
		\bigskip
		\small
		
				\textsc{Institut Montpellierain Alexander Grothendieck, 499-554 Rue du Truel, 34090 Montpellier, France.}\par\nopagebreak
		\textit{E-mail address}: \texttt{anthony.genevois@umontpellier.fr}
\medskip

		\textsc{Department of Mathematics, University of the
			Basque Country UPV/EHU, Leioa, Spain. IKERBASQUE, Basque Foundation for Science, Bilbao, Spain.}\par\nopagebreak
		\textit{E-mail address}: \texttt{annemarguerite.lonjou@ehu.eus}
\medskip

		\textsc{EPFL, SB MATH, Station 8, CH-1015 Lausanne, Switzerland.}\par\nopagebreak
\textit{E-mail address}: \texttt{christian.urech@epfl.ch}
\medskip
		
}}


\title{Cremona groups over finite fields, Neretin groups, and non-positively curved cube complexes}
\date{\today}
\author{Anthony Genevois, Anne Lonjou, and Christian Urech}


\begin{document}

\maketitle
\begin{abstract}
We show that plane Cremona groups over finite fields embed as dense subgroups into Neretin groups, i.e.\ groups of almost automorphisms of rooted trees. We also show that if the finite base field has even characteristic and contains at least 4 elements, then the permutations induced by birational transformations on rational points of regular projective surfaces are even. 

In a second part, we construct explicit locally compact CAT(0) cube complexes, on which Neretin groups act properly. This allows us to recover in a unified way various results on Neretin groups such as that they are of type $F_{\infty}$. We also  prove a new fixed-point theorem for CAT(0) cube complexes without infinite cubes and use it to deduce a regularisation theorem for plane Cremona groups over finite fields.

\end{abstract}
\tableofcontents
\section{Introduction}

In this article, our goal is to bring together two a priori quite distinct families of groups. Namely, \emph{Cremona groups}, i.e.\ the groups  of birational transformations of projective planes, and \emph{Neretin groups}, i.e.\ the (totally disconnected locally compact) groups of almost automorphisms {$\AAut(\mathcal{T})$} of regular rooted trees {$\mathcal{T}$}. While the first come from algebraic geometry, the second appear in low-dimensional topology and belong to the more general family of \emph{Thompson-like groups}. Neretin groups can be thought of as $p$-adic analogues of the diffeomorphism group of the circle.

\medskip
At first glance, these families of groups seem quite different, and they are from several perspectives. For instance, algebraically speaking, Neretin groups are (uniformly) simple \cite{MR1703086, MR3693109} while Cremona
groups are very far from being simple, since they are acylindrically hyperbolic, which allows us to construct many different normal subgroup (\cite{Cantat-Lamy}, \cite{Lonjou_non_simplicity}). Nevertheless, they both arise as groups of transformations that only partially preserve certain geometric structures. Our goal is to exploit this analogy by transferring techniques from one family of groups to the other, giving new insights for both types of groups.

\paragraph{Neretin groups.} 
Our main contribution to Neretin groups $\mathcal{N}_d:=\AAut(\T_{d})$, where $\T_d$ is the regular rooted tree of degree $d \geq 2$, is the construction of locally compact CAT(0) cube complexes on which they act properly {(as topological groups)}. Even if the existence of proper actions on (a priori not locally compact) CAT(0) cube complexes was already known (see \cite[Section~3.3.3]{le2015geometrie}), transferring the cubulation of Cremona groups from \cite{lonjouthesis, lonjou-urech} into the world of Neretin groups allows us to construct explicit cube complexes, whose geometric structures are tightly connected to the algebraic structures of Neretin groups.

\begin{theorem}\label{thm:IntroCubulation}
For every $d \geq 2$, the Neretin group $\mathcal{N}_d$ acts properly on a locally compact CAT(0) cube complex. Moreover, every vertex-stabiliser equals the automorphism group of some cofinite subforest and, conversely, for each cofinite  subforest, there exists a vertex-stabiliser that is its automorphism group.
\end{theorem}

This construction allows us to recover several results proved separately in the literature in a unified way. Namely, we deduce from Theorem~\ref{thm:IntroCubulation} that Neretin groups are a-T-menable; that their subgroups satisfying Kazhdan's property (T) (or more generally the fixed-point property $(\mathrm{FW}_\mathrm{locfin})$, see Section~\ref{section:FirstApplications}) lie in automorphism groups of cofinite rooted subforests; and that they are of type $F_\infty$ (a finiteness property stronger than being compactly presented, see Section~\ref{section:FirstApplications}). The local finiteness of our cube complexes can also be exploited, allowing us to recover the following result from \cite{le2018commensurated}:

\begin{theorem}\label{thm:IntroPurelyEllipticNeretin}
Let $H \leq \mathcal{N}_d$ be a finitely generated subgroup. If each element of $H$ induces an automorphism of some cofinite subforest, then $H$ entirely lies in the automorphism group of a cofinite subforest. 
\end{theorem}

Theorem~\ref{thm:IntroPurelyEllipticNeretin} is proved by combining Theorem~\ref{thm:IntroCubulation} together with the following fixed-point theorem  from cubical geometry, which is of independent interest. 

\begin{theorem}\label{thm:IntroFixedPointThm}
Let $G$ be a finitely generated group acting on a CAT(0) cube complex $X$ by elliptic isometries. If $X$ has no infinite cube, then $G$ has a global fixed point.
\end{theorem}

Here, an infinite cube refers to {the union of } an infinite increasing sequence $C^0\subset C^1\subset\cdots$ of cubes, where each cube $C^n$ is a combinatorial unit cube with $2^n$ vertices.
		
\medskip

Theorem~\ref{thm:IntroFixedPointThm} extends the fixed-point theorem from \cite{Sageev-ends_of_groups} for finite-dimensional CAT(0) cube complexes to a large class of infinite-dimensional CAT(0) cube complexes. The conclusion of Theorem \ref{thm:IntroFixedPointThm} does not always hold if we drop the assumption that there is no infinite cube. Indeed, there exist finitely generated infinite torsion groups acting without global fixed points on CAT(0) cube complexes (such as the Grigorchuk group \cite{MR3027509, Sageev-ends_of_groups}, Burnside groups \cite{Osajda_2018}, or just wreath products of torsion groups \cite{cornulier_wallings, WreathFW, LampG}).

\paragraph{Cremona groups.} Although most of our constructions make sense over arbitrary fields, their main applications concern Cremona groups over finite fields $\kk$ {denoted by $\Bir_\kk(\p^2)$ instead of $\Bir(\p^2)$ for readers' convenience. }  This is a subject that has attracted substantial interest recently, also because of its connections to cryptography (see for instance \cite{cornulier2013sofic}, \cite{shepherd2021some}, \cite{schneider2020generators}, \cite{schneider2020algebraic}, \cite{Zimmermann-bir_permutations}, \cite{lamy2021generating}). In this work we construct a rooted forest $\FF\p^2(\kk)$, {called the \emph{ rational blow-up forest},} whose vertices are the $\kk$-rational points of $\p^2$ and of all the surfaces obtained by blowing up $\kk$-rational points. One of the reasons why birational transformations of surfaces are well understood, is the fact that they can be factorized into a sequence of blow-ups of points and contractions of curves. This property allows us to construct a natural  injective morphism from $\Bir_\kk(\p^2)$  to the group of almost automorphisms of $\FF\p^2(\kk)$. If the field $\kk$ is finite with $q$ elements, then the rooted forest  $\FF\p^2(\kk)$ turns out to be almost isomorphic to the $(q+1)$-regular rooted tree.
This will lead to a proof of the following result, which displays the close relationship between Cremona groups over finite fields and Neretin groups:

 \begin{theorem}\label{thm:MainNeretinCremona}
 	If $\kk$ is a finite field with $q$ elements, then $\Bir_\kk(\p^2)$ is isomorphic to a dense subgroup of the Neretin group $\mathcal{N}_d$, where $d=q+1$.  
 \end{theorem}

Note that, even if recently the use of tools from geometric group theory in the study of groups of birational transformations of surfaces has been really fruitful, this is the first time that these groups appear as almost automorphism groups of trees or forests.
\medskip

The topology on $\AAut(\T_{d})$  induces a non-discrete topology on  $\Bir_\kk(\p^2)$ that turns it into a Hausdorff topological group {(which is not locally compact)}. In fact,  
Theorem~\ref{thm:MainNeretinCremona} can also be seen as an analogue to a theorem of Koll\'ar and Mangolte, which states that the group of birational transformations of the real projective plane without real indeterminacy points is a dense subgroup of the diffeomorphism group of $\p^2(\R)$ (\cite{kollar2009cremona}). In Remark~\ref{rem:topology} we give a description of sequences of elements in $f\in\Bir_\kk(\p^2)$ converging to the identity. 

\medskip

Our point of view provides a transparent picture about how $\Bir_\kk(\p^2)$ acts on $\kk$-rational points on surfaces, which is displayed in the following application to birational geometry of surfaces.
Let $X$ and $Y$ be varieties over $\kk$, by which we mean integral and separated schemes of finite type. Let $f\colon X\dashrightarrow Y$ be a birational transformation and assume that the indeterminacy loci of $f$ and of $f^{-1}$ do not contain any $\kk$-rational point. Then $f$ induces a bijection between the sets of $\kk$-rational points $X(\kk)$ and $Y(\kk)$. We will call such a map a birational map that is \emph{bijective on $\kk$-rational points}. We denote by $\Bir(X)$ the group of birational transformations of $X$ and by $\BBir(X)$ the group of birational transformations that are bijective on $\kk$-rational points. The group $\BBir_\kk(\p^2)$ for finite fields $\kk$ has first been considered by Cantat in \cite{Cantat-Bir_permut}, where he showed that in odd characteristic and in the case $\kk=\F_2$, every permutation on the $\kk$-rational points $\p^2(\kk)$ is induced by an element from $\BBir_\kk(\p^2)$. We will prove the following result:

\begin{theorem}\label{thm:mainparity}
	Let $\kk$ be a finite field of even characteristic and assume that $\kk\neq\F_2$. If $S$ is a regular projective rational surface over $\kk$ and $f\in \BBir(S)$, then the induced permutation on the $\kk$-rational points $S(\kk)$ is even. 
\end{theorem}

The strategy of our proof consists in using the action of $\Bir_\kk(\p^2)$ by almost automorphisms on $\FF\p^2(\kk)$ to extend the notion of parity from $\BBir_\kk(\p^2)$ to all of $\Bir_\kk(\p^2)$. We can then deduce the theorem for $\p^2$ from a recent result by Lamy and Schneider (Theorem~\ref{thm:lamyschneider}) and finally generalize it to arbitrary $S$. 
\medskip

We would like to highlight the fact that over finite fields there exist non-trivial homomorphisms from $\Bir_\kk(\p^2)$ to $\Z/2\Z$ (\cite{lamy2020signature}), so  Theorem \ref{thm:mainparity} is not a plain consequence of  the abstract group structure of $\Bir_\kk(\p^2)$.  

\begin{remark}
	The fact that if $|\kk|=2^n\geq 4$, then the permutations on the $\kk$-rational points $\p^2(\kk)$ induced by elements from $\BBir_\kk(\p^2)$ are all even, was conjectured in the preprint \cite{Zimmermann-bir_permutations}.
	We learned that in parallel to our work, Asgarli, Lai, Nakahara, and Zimmermann came up independently of us with a proof of this conjecture, which they include in the published version of their manuscript \cite{Zimmermann-bir_permutations}. Their methods are different from ours: they study the group $\BBir_\kk(\p^2)$ in detail, provide an explicit generating set, and show by hand that every element in the generating set induces an even permutation. In our paper, we use a result (Theorem~\ref{thm:finiteorder}) from \cite{Zimmermann-bir_permutations}. However, our main strategy for the proof of Theorem~\ref{thm:mainparity} is different and our result is more general, since it shows that the permutations of $S(\kk)$ induced by $\BBir(S)$ are even for arbitrary regular rational projective surfaces $S$.
\end{remark}

In \cite{lonjou-urech}, the second and third author constructed for every surface $S$ over any field $\kk$, the \emph{blow-up complex} - a $\CAT(0)$ cube complex, on which $\Bir(S)$  acts by isometries - and used it to deduce various dynamical and group theoretical properties. However, one of the drawbacks of this construction is that  these cube complexes are never locally compact. If we work over a finite field, we can now use the locally compact cube complexes constructed for the Neretin groups to obtain a locally compact $\CAT(0)$ cube complex on which $\Bir(S)$ acts by isometries.
\medskip

{In \cite{lonjou-urech}, we asked the following question: Assume that in a finitely generated subgroup $\Gamma\subset\Bir(S)$  every element is conjugate to an automorphism of a projective surface, does this imply that $\Gamma$ itself is conjugate to a subgroup of automorphisms of a projective surface? This question seems to be subtle and difficult in general. However, in the case, where $\kk$ is a finite field, we obtain a positive answer, by applying Theorem~\ref{thm:IntroFixedPointThm} to the action of $\Bir(S)$ on the locally compact complex given by the Neretin groups.}

\begin{theorem}\label{thm:CremonaReg}
	Let $\kk$ be a finite field, $S$ a surface over $\kk$, and $\Gamma\subset\Bir(S)$ a finitely generated subgroup such that every element in $\Gamma$ is conjugate to an automorphism of a projective surface, then $\Gamma$ itself is conjugate to a subgroup of automorphisms of a projective surface.
\end{theorem}

In a very similar spirit, we also prove that if $S$ is a regular projective surface over a finite field $\kk$ and if $\Gamma\subset\Bir(S)$ is a finitely generated subgroup such that for every element $\gamma\in\Gamma$ there exists a regular projective surface $S'$ over $\kk$ such that $\gamma$ is conjugate to an element in $\BBir(S')$, then there exists a regular projective surface $T$ over $\kk$ such that $\Gamma$ is conjugate to a subgroup of $\BBir(T)$ (see Proposition~\ref{prop:BBirreg}).

\subsection*{Outline of the article}
After recalling some preliminaries, we will construct in Section~\ref{sec:CremonaNeretin} the rational blow-up forest, which is the key construction in order to embed plane Cremona groups over a finite field into Neretin groups. The embedding and density of Cremona groups in Neretin groups (Theorem~\ref{thm:MainNeretinCremona}) is then proved in Section~\ref{Subsection_embedding} and the result about the parity of Cremona transformations (Theorem \ref{thm:mainparity}) in Section~\ref{Subsection_parity}. Section~\ref{Section_cubulation_Neretin} is devoted to the construction of the CAT(0) cube complex and its applications. The proof of the fixed point Theorem~\ref{thm:IntroFixedPointThm} finally is done in Section~\ref{Subsection_fixed_point_theorem}.

\subsection*{Acknowledgements}The authors would like to thank Yves Cornulier and Waltraud Lederle for helpful comments on the paper, as well as the anonymous referee for many suggestions that helped improve the exposition of the article. The second author is grateful to the CNRS for its support through the grant PEPS JC/JC.

\section{Preliminaries}
In this section we briefly recall some definitions and some results about Cremona groups and Neretin groups.
\subsection{Cremona groups}\label{sec:cremona}
In this section we recall some results from the birational geometry of surfaces, we refer to  \cite{Cantat_review} or \cite{lamycremona} and the references therein for details about Cremona groups, and to \cite{poonen2017rational} for details about rational points on varieties. In this article, a variety over a field $\kk$ is always an integral and separated schemes of finite type. When speaking about a morphism or a birational map between varieties over a field $\kk$, we always mean a $\kk$-morphism, or a $\kk$-birational map (unless stated otherwise). 
\medskip

Let $X$ be a variety over a field $k$. Recall that a \emph{$\kk$-rational point}, or just \emph{rational point}, is a morphism $\Spec k\to X$. The set of all rational points of $X$ is denoted by $X(\kk)$. A morphism between $k$-varieties $X\to Y$ induces a map $X(\kk)\to Y(\kk)$.

\medskip
Let $S $ be a regular surface over $\kk$ and let $p$ be a closed point on $S$. The \emph{blow-up} of the surface $S$ in $p$ is a variety $Bl_p$ together with a morphism $\pi\colon Bl_p\to S $ such that the inverse image of $p$ is a Cartier divisor, which is called \emph{the exceptional divisor of $\pi$}, and such that the following universal property is satisfied: For a morphism $\pi'\colon S '\to S $ such that the inverse image of $p$ is a Cartier divisor, there always exists a unique morphism $f\colon S '\to Bl_p$ that satisfies $\pi'=\pi f$. Blow-ups always exist and moreover, if $S $ is a regular projective surface, then $Bl_p$ is a regular projective surface. The exceptional divisor is isomorphic to $\p^1_L$, where $L=k(p)$ is the function field of $p$. In particular, if $p$ is a rational point, then the exceptional divisor is isomorphic to $\p^1_\kk$. Let us also note that the exceptional divisor contains rational points if and only if $p$ is rational.

\medskip
Let $S$ be a regular projective surface over $\kk$. The {\it bubble space} $\B S$ of the surface $S$ is, roughly speaking, the set of all closed points that belong to $S$ or are infinitely near to $S$, that is, contained on a surface dominating $S$. More precisely, $\B S$ is the set of all triples $(y,T,\pi)$, where $T$ is a regular projective surface over $\kk$, $y$ is a closed point in $T$ and $\pi\colon T\to S$ is a birational morphism, modulo the following equivalence relation: A triple $(y, T, \pi)$ is equivalent to $(y', T', \pi')$ if the birational map $\pi'^{-1}\pi\colon T\dashrightarrow T'$ induces an isomorphism in a neighbourhood  of $y$ and maps $y$ to $y'$. A point $p\in\B S$ that is equivalent to some $(x,S,\id)$ is called a {\it proper point} of $S$. All points in $\B S$ that are not proper are called {\it infinitely near}. If there is no ambiguity, we will sometimes denote a point in the bubble space by $y$ instead of $(y, T, \pi)$.
A point $p\in\B S$ is a \emph{rational} point, if $p$ can be represented by a triple $(y, T, \varphi)$ such that $y$ is a rational point on $T$. We denote the set of rational points in $\B S$ by $\B S(\kk)$.

\medskip
Any birational transformation between projective surfaces can be factored into blow-ups of closed points:

\begin{theorem}[{\cite[\href{https://stacks.math.columbia.edu/tag/0C5Q}{Tag 0C5Q, Tag 0C5J}]{Stack_project}}]\label{factorization}
	Let $S $ and $S '$ be regular projective surfaces over a field $\kk$ and let $f\colon S \dashrightarrow S '$ be a birational transformation. Then there exists a projective surface $\tilde{S}$ over $\kk$ together with two morphisms $\eta\colon \tilde{S} \to S $, $\rho\colon \tilde{S} \to S '$ satisfying $f=\rho\eta^{-1}$, such that we can factorize $\eta\colon \tilde{S} \to S ^n\to\cdots\to S ^1\to S ^0=S $ and $\rho\colon \tilde{S} \to {S'} ^m\to\cdots\to {S'}^1\to {S'}^0=S' $, where each of the arrows is a blow-up in a closed point.
	
	Moreover, $\tilde{S}$ can be chosen minimal in the following sense: for any other projective regular surface $\tilde{S}'$ giving rise to such a factorization, there exists a surjective morphism $\pi\colon  \tilde{S}' \rightarrow  \tilde{S}$. This implies in particular that $ \tilde{S} $ and $ \tilde{S}' $ are respectively obtained from $S $ by blowing up a unique sequence of points in the bubble space $\B S $.
\end{theorem}

As a particular instance of Theorem~\ref{factorization} we have that if $\pi\colon S'\to S$ is a morphism between regular projective surfaces, then $\pi$ is a blow-up of some points in $\B S$. In this case, we say that $S'$ lies \emph{above} $S$. 

\medskip
Theorem~\ref{factorization} allows us to define the notion of base-points:

\begin{definition}
	Let $f\colon S \dashrightarrow S' $ be a birational transformation between two regular projective surfaces. The \emph{base-points} of $f$ are the points in $\B S$ that are blown up by $\eta$  in the minimal resolution of $f$. The set of base-points of $f$ will be denoted by $\B(f)$.
\end{definition}

A birational morphism $\pi\colon S\to T$ of regular projective surfaces $S$ and $T$ induces a bijection $\B S\to\B T\setminus\B(\pi^{-1})$ by mapping a point represented by $(x,{S'},\varphi)$ to the point represented by $(x,{S'}, \pi\circ\varphi)$. {By abuse of notation we will denote this map by $\pi$ as well.} Note that $\pi$ maps  rational points to  rational points. {By Theorem~\ref{factorization}}, a birational transformation of regular projective surfaces $f\colon S\dashrightarrow T$ induces a bijection  $f\colon\B S\setminus\B(f)\to\B T\setminus \B(f^{-1})$ by   {$f\coloneqq \rho\circ\eta^{-1}$, where $\eta$ and $\rho$ are given by Theorem \ref{factorization} for a minimal resolution}. In particular, $f$ maps  rational points that are not base-points of $f$ bijectively to  rational points that are not base-points of $f^{-1}$.

\subsection{Neretin groups}\label{sec:neretin}

In this section we recall the construction of the Neretin groups as almost automorphism groups of rooted trees, as well as their topologies. For details we refer to \cite{garncarek2015neretin, neretin1992combinatorial, MR1703086, le2018commensurated}.

\paragraph{Almost isomorphisms of graphs.}Given two (not necessarily connected) graphs $X$ and $Y$, an \emph{almost isomorphism} $(f,R,S)$ is the data of two finite subgraphs $R \subset X$, $S \subset Y$ and a graph isomorphism $f : X \backslash R \to Y \backslash S$. In order to shorten the notation, one might write $f : X \dashrightarrow Y$ without referring to $R,S$. Two almost isomorphisms are considered as equal if they agree on some cofinite subset of vertices. We denote by $\mathrm{AAut}(X)$ the almost automorphism group of the graph $X$. Observe that modifying $X$ in a finite subset does not modify its almost automorphism group. More precisely, if $f : X \dashrightarrow Y$ is an almost isomorphism, then the map $g \mapsto fgf^{-1}$ induces an isomorphism $\mathrm{AAut}(X) \to \mathrm{AAut}(Y)$. 

\begin{definition}
Given two integers $d \geq 2$ and $r \geq 1$, the \emph{Neretin group} $\mathcal{N}_{d,r}$ is the almost automorphism group $\mathrm{AAut}(\mathcal{T}_{d,r})$ of the rooted tree $\mathcal{T}_{d,r}$ whose root has degree $r$ and all of whose other vertices have degree $d+1$. 
\end{definition}

\noindent
Observe that, as a consequence of the previous remark, $\mathcal{N}_{d,r}$ also coincides with the almost automorphism group of the disjoint union of $r$ rooted $d$-regular trees. For simplicity, we note $\mathcal{T}_{d}:= \mathcal{T}_{d,1}$ and $\mathcal{N}_d:= \mathcal{N}_{d,1}$.

\paragraph{Almost automorphisms of rooted forests.} In this article, we are mainly concerned with almost automorphisms of rooted trees or forests, which can be described in a more convenient way for our purpose. 

\medskip \noindent
So let $\mathcal{F}$ be a rooted forest. {We denote by $\lambda(\mathcal{F})$ the set of roots of the forest $\mathcal{F}$.} A rooted subforest $F \subset \mathcal{F}$ is \emph{admissible} if its roots are contained in the roots of $ \mathcal{F}$ and if it is finite (possibly empty). {The cofinite subforest $F^C:=\mathcal{F}\setminus F$ of $\mathcal{F}$ is obtained from $\mathcal{F}$ by removing the vertices that are in $F$ and the edges that have a least an endpoint in $F$. }  {Notice that the cofinite subforest $F^C$ is rooted in a natural way, where the roots are given by roots of $\mathcal{F}$ and leaves of $F$.}

\medskip \noindent
Then an almost automorphism of the forest $\mathcal{F}$ can be represented as an equivalence class of triples $(\psi, F, F')$, where $F$ and $F'$ are admissible rooted subforests of $\mathcal{F}$, and where $\psi\colon F^C \to F'^C$ is a rooted forest isomorphism, i.e.\ a forest isomorphism sending bijectively $\lambda( F^C)$ to $\lambda(F'^C)$. Two such triples $(\psi_1, F_1,F_1')$ and $(\psi_2, F_2, F_2')$ are equivalent if there exist admissible rooted subforests $F \supset F_1,F_2$ and $F' \supset F_1',F_2'$ such that $\psi_1$ and $\psi_2$ induce the same rooted forest isomorphism $F^C\to F'^C$.

\medskip Following \cite{le2018commensurated}, we call an almost automorphism $g$ of a forest $\mathcal{F}$ {\it elliptic}, if there exists an admissible subforest $F$ such that $g$ can be represented by a triple $(\psi, F, F)$.

\paragraph{Topologies on the Neretin groups}
Let $\mathcal{T}$ be a rooted tree. A \emph{ray} departing from the root $v_0$ of $\mathcal{T}$ is a sequence of distinct vertices $(v_0,\ v_1,\dots)$ such that $v_i$ and $v_{i+1}$ are connected by an edge for all $i\geq 0$. The \emph{visual boundary} $\partial \mathcal{T}$ is the space of all rays departing from $v_0$ equipped with the \emph{visual metric}, i.e.\ for $\eta_1, \eta_2\in\partial \mathcal{T}$ we define $d(\eta_1, \eta_2)=e^{-\delta(\eta_1, \eta_2)}$, where $\delta(\eta_1, \eta_2)$ denotes the length of the common initial path of the rays $\eta_1$ and $\eta_2$. The visual metric defines an ultrametric on $\partial \mathcal{T}$, which turns it into a compact and second countable space that is totally disconnected. Let us note that a closed ball in $\partial \mathcal{T}$ corresponds exactly to the rays passing through a given vertex in $\mathcal{T}$.  
\medskip

The homeomorphism group $\Homeo(\partial \T_{d,r})$ can be equipped with the compact-open topology, i.e.\ the topology induced by the following metric: for all $f, g\in\Homeo(\partial \T_{d,r})$, we define \[d(f,g):=\max_{\eta\in\partial\T_{d,r}}d(f(\eta), g(\eta)).\] We now equip $\Aut(\T_{d,r})$ with the restriction of this topology. Similarly, we could consider $\AAut(\T_{d,r})$ with the induced compact-open topology. However, with this choice, the group $\AAut(\T_{d,r})$ would not be locally compact, since $\AAut(\T_{d,r})\subset \Homeo(\partial \T_{d,r})$ is not closed. For this reason, $\AAut(\T_{d,r})$ is usually equipped with the unique topology that turns it into a topological group such that the injection $\Aut(\T_{d,r})\to \AAut(\T_{d,r})$ is continuous and open. With respect to this topology, 
$\AAut(\T_{d,r})$ is a totally disconnected locally compact group (see \cite{garncarek2015neretin} for details and proofs).

\begin{remark}
	Let $S_1, S_2\subset \T_{d}$ be two admissible subtrees. Then the set of all elements $g\in \AAut(\T_{d})$ that have a representative of the form $(\varphi, S_1, S_2)$ is open. Moreover, the open sets of this type form a basis for the Neretin group.
\end{remark}

\medskip \noindent
For a more topological interpretation of almost automorphisms of $\T_{d,r}$, let us observe that an almost automorphism induces a homeomorphism of the visual boundary $\partial \T_{d,r}$. The subgroup $\Aut(\T_{d,r})\subset \AAut(\T_{d,r})$ of automorphisms of the rooted tree $\T_{d,r}$ corresponds exactly to the isometries of $\partial\T_{d,r}$. Almost automorphisms correspond exactly to the homeomorphisms $\varphi$ of $\partial \T_{d,r}$ such that there exists a partition of $\partial\T_{d,r}$ into disjoint closed balls $B_1,\dots, B_n$ and the restriction of $\varphi$ to $B_i$ is a homothety for all $1 \leq i \leq n$. This justifies the common term \emph{spheromorphism} for elements in the Neretin group.

\section{Cremona groups in Neretin groups}\label{sec:CremonaNeretin}

\subsection{The blow-up forest}\label{Section_blow_up_tree}
Let $\kk$ be an arbitrary field.
We can equip the bubble space of a regular projective surface $S$ with the structure of a forest, the \emph{blow-up forest} $\mathcal{F} S$. The vertices of $\mathcal{F} S$ are the points in the bubble space $\B S$ and two distinct vertices $p$ and $q$ are connected by an edge if we can write $p=(x, T, \varphi)$ and $q=(y, T', \psi)$ such that $\varphi^{-1}\psi\colon T'\to T$ is the blow-up in the point $x$ and $y$ is contained in the exceptional divisor of this blow-up. 
We define the \emph{rational blow-up forest} $\mathcal{F}S(\kk)$ as the subforest induced by the vertices belonging to $\B S(\kk)$. To every blow-up $\pi\colon S'\to S$ we can associate the rooted subforests $\FF(\pi)\subset\mathcal{F}S$ and $\FF(\pi)(\kk)\subset\mathcal{F}S(\kk)$ induced by the (rational) base-points of $\pi{^{-1}}$. Note that the closed points of $S'$ are {the roots of $\FF(\pi)^C$ and the rational points of $S'$ are the roots of $\FF(\pi)(\kk)^C$.}
\medskip

Conversely, to an admissible rooted subforest $F\subset\mathcal{F}S$, we can associate a blow-up $\pi_F\colon {S_F}\to S$, where $\pi_F$ blows up the points in the bubble space corresponding to the vertices of $F$.

\medskip

Let $f\in\Bir(S)$ be a birational transformation. We observe that the induced bijection $f\colon \B S\setminus\B(f)\to \B S\setminus \B(f^{-1})$ constructed in Section~\ref{sec:cremona} preserves the forest structure and hence induces an almost automorphism of the blow-up forest $\FF S$, which we denote by $\tilde{f}$. Since $f$ preserves rational points, it also induces an almost automorphism of the rational blow-up forest $\FF S(\kk)$, which will be denoted by $\tilde{f}$ as well.

This proves the following proposition:

\begin{proposition}\label{prop:neretin}
	Let $\kk$ be a field and $S$ a regular projective surface over $\kk$. Then $\Bir(S)$ acts by almost automorphisms on the blow-up forest $\FF S$, i.e.\ there exists a morphism from $\Bir(S)$ to $\AAut(\FF S)$, as well as on the rational blow-up forest $\FF S(\kk)$. 
\end{proposition}

In what follows we will give a  description of $\tilde{f}$ and introduce some vocabulary needed in the sequel. We focus on the rational blow-up forest, but everything can be defined analogously for the blow-up forest.
To a birational transformation $f\in\Bir(S)$ we associate the {(rational)}\emph{base-point forest} $\FF f$ {respectively $\FF f(\kk)$}, which we define to be the admissible rooted subforest of $\FF S$ {(respectively $\FF S(\kk)$ )} induced by the {(rational)} base-points $\B(f)$ of $f$. 

\medskip

For an $f\in\Bir(S)$, consider the rational base-point forests $F:=\FF f(\kk)$ and $F':=\FF f^{-1}(\kk)$ and let $\pi_F\colon S_F\to S$ and $\pi_{F'}\colon S_{F'}\to S$ be blow-ups of the corresponding points in the bubble space $\B S(\kk)$. Then 
\[
\pi_{F'}^{-1}f\pi_F\colon S_{F}\dashrightarrow S_{F'}\]
is bijective on $\kk$-points, so it induces a rooted forest isomorphism $\mathcal{F} S_{F}(\kk)\to \mathcal{F} S_{F'}(\kk)$. If we identify $\mathcal{F} S_F(\kk)$ with $F^C$ through $\pi_F$ and $\mathcal{F} S_{F'}(\kk)$ with $F'^C$  through $\pi_{F'}$ respectively, this yields a rooted forest isomorphism 
\[
\psi_{f}\colon  F^C\to F'^C,\]
and we can write $\tilde{f}=\left(\psi_{f}, F, F'\right)$. Note that $\psi_{f}$ does not depend on the choice of the blow-ups $\pi_F$ and $\pi_{F'}$.

\begin{remark}
	Although we can represent $f\in \Bir(S)$ by the triple $(\psi_f, \FF f(\kk), \FF f^{-1}(\kk))$, the admissible forests $\FF f^{-1}(\kk)$ and  $\FF f(\kk)$ are not necessarily the minimal forests that appear in a representative of $\tilde{f}$. For instance, assume that $\kk$ is finite. Let $\pi\colon S'\to \p^2$ be a blow-up of sufficiently many rational points and let $f\in\Bir_\kk(\p^2)$ be an element that is conjugate to an automorphism of $S'$ of infinite order. Up to iterating $f$, we may assume that this automorphism fixes all the rational points $S'(\kk)$, and as a consequence the almost automorphism $\tilde{f}$ is in fact an automorphism of $\FF \p^2$. However, having infinite order, $f$ is not conjugate to an automorphism of $\p^2$, in particular, $\FF f(\kk)$ and $\FF f^{-1}(\kk)$ are non-empty.
\end{remark}

\begin{lemma}\label{lem:nobasepoints}
Let $S$ be a regular projective surface over $\kk$. If $f\in\Bir(S)$ induces an almost automorphism $\tilde{f}$ on $\FF S(\kk)$ that has a representative of the form $(\varphi, F,F')$ such that {the vertices in the outer boundary} of $F$ do not contain any base-point of $f$ and {the vertices in the outer boundary} of $F'$ do not contain any base-points of $f^{-1}$, then $f$ induces a birational map $S_F\dashrightarrow S_{F'}$ that is bijective on $\kk$-points.
\end{lemma}

\begin{proof}
	The rational base-point forests $\FF f(\kk)$ and $\FF f^{-1}(\kk)$ have to be contained in $F$ and $F'$, respectively. If $F$ contains one more vertex than $\FF f(\kk)$, then we obtain $S_F$ by blowing up a point $p$ in $S_{\FF f(\kk)}$. Correspondingly, $S_{F'}$ is obtained by blowing up the point $\psi_f(p)$ in $S_{\FF f^{-1}(\kk)}$. Therefore, the induced birational map $S_F\dashrightarrow S_{F'}$ maps the exceptional divisor of the blow-up of $p$ isomorphically to the exceptional divisor of the blow-up of $\psi_f(p)$. It is therefore bijective on $\kk$-points. We proceed inductively to obtain the general case. 
\end{proof}

In a next step, let us observe that over finite fields birational transformations that are conjugate to an element in $\BBir(S')$ for some regular projective $S'$, correspond exactly to the birational transformations inducing elliptic almost automorphisms:

\begin{lemma}\label{lem:CremonaElliptic}
Let $\kk$ be finite and $S$ a regular projective surface over $\kk$. Let $\Gamma\subset \Bir(S)$ be a finitely generated subgroup. Then there exists an {admissible} rooted subforest $F\subset\FF S(\kk)$ such that every element in $\Gamma$ can be represented by a triple of the form $(\psi, F, F)$ if and only if $\Gamma$ is conjugate to a subgroup of $\BBir(S')$ for some regular projective surface $S'$ {over $\kk$}.
\end{lemma}

\begin{proof}
	First assume that there exists an {admissible} rooted subforest $F\subset\FF S(\kk)$, such that every $g\in\Gamma$ can be represented by a triple of the form $(\psi, F, F)$ for some rooted forest isomorphism $\psi\colon F^C\to F^C$. Let $F^n$ be the rooted admissible subforest obtained from $F$ by adding all vertices of $ F^C$ at distance at most $n-1$ of  the set of roots of $\lambda(F^C)$. Observe that every element in $\Gamma$ can be represented by the triple $(\psi|_{{(F^n)}^C},F^n, F^n)$. For $n$ large enough, $F^n$ contains all the vertices corresponding to {rational} base-points of elements of a {symmetric} generating set $\{\gamma_1,\dots, \gamma_m\}$ of $\Gamma$. As a consequence, Lemma~\ref{lem:nobasepoints} implies that $\Gamma$ is conjugate to a subgroup of $\BBir(S_{F^n})$.
	
	\medskip
	On the other hand, assume that $\Gamma$ is conjugate to a subgroup of $\BBir(S')$ for some regular projective surface $S'$ {over $\kk$}. After blowing up enough {rational} points we may assume that $S'$ lies above $S$ and that $\Gamma$ is conjugate to a subgroup of $\BBir(S')$ by a blow-up $\pi\colon S'\to S$ of rational points. Let $F{\subset \FF S(\kk)}$ be the admissible rooted subforest induced by the points blown up by $\pi$. Then every element $f\in\Gamma$ induces the  {almost automorphism} represented by the triple $(\psi_f,F,F)$, where $\psi_f$ is the forest isomorphism induced by $\pi^{-1}f\pi\in\BBir(S')$.
\end{proof}

\subsection{Embedding and density}\label{Subsection_embedding}

In this section, our base-field $\F_q$ is the field with $q=p^r$ elements. We will consider the Cremona group $\Bir_{\F_q}(\p^2)$ over $\F_q$. Note that a birational map $\varphi\colon \p^2\dashrightarrow S$ to a regular projective surface $S$ over $\F_q$ induces  an isomorphism between $\Bir_{\F_q}(\p^2)$ and $\Bir(S)$ by conjugation. Also recall that a geometrically rational regular projective surface $S$ over $\kk$ is birationally equivalent to $\p^2$ if and only if $S$ admits at least one rational point. 
\begin{lemma}\label{lem:faithful}
	 The morphisms from $\Bir_{\F_q}(\p^2)$ to $\AAut(\FF\p^2(\F_q))$ and to $\AAut(\FF\p^2)$ given by Proposition\ref{prop:neretin} are injective.
\end{lemma}

\begin{proof}
It is sufficient to show that the morphism from $\Bir_{\F_q}(\p^2)$ to $\AAut(\FF\p^2(\F_q))$ is injective.
	Let $f\in\Bir_{\F_q}(\p^2)$ be a transformation that induces the identity, as almost automorphism, on $\FF\p^2(\F_q)$, i.e.\ $f$ fixes all but finitely many rational points in the bubble space $\B \p^2(\F_q)$. We show that $f$ is the identity. After blowing up all the rational base-points of $f$, we obtain a regular projective surface $S$ such that $f\in\BBir(S)$. We will work now over the algebraic closure of $\F_q$. Observe that, as $f$ is a local isomorphism around each point $p\in S{(\F_q)}$, $f$ induces an automorphism on the exceptional divisor $E_p$ of the blow-up of $p$, which is the identity, since $E_p$ contains $q+1\geq 3$ rational points.  Let $p\in S(\F_q)$ and let $U\subset S$ be a neighbourhood of $p$ isomorphic to $\A^2$ with local coordinates $(x,y)$ such that $p=(0,0)$. Consider the curve $C$ given by $\{x=0\}$ and denote by $f(C)$ the strict transform of $C$. 
	\medskip
	
	Let $\pi_p\colon\tilde{S}\to S$ be the blow-up of $p$ and $\tilde{f}$ the birational transformation on $\tilde{S}$ induced by $f$. Locally, $\pi_p$ is given by $(x',y')\to (x'y', y')$. The strict transform $\tilde{C}$ of $C$ under $\pi_p$ is given by $\{x'=0\}$. Assume that $f(C)$ is given by the local equation $p(x,y)=p_m(x,y)+r(x,y)=0$, where $p_m$ is a homogeneous polynomial of degree $m$ and $r(x,y)\in (x,y)^{m+1}$. The total transform of $f(C)$ under $\pi_p$ is therefore given by $y'^m(p_m(x', 1)+s(x',y'))$ for some $s(x',y')$ divisible by $y'$, and the strict transform $\tilde{f}(\tilde{C})$ of $f(C)$ under $\pi_p$ is given by $p_m(x', 1)+s(x',y')=0$. Since $\tilde{f}$ is a local isomorphism around $E_p$ and fixes every point on $E_p$, the strict transform $\tilde{f}(C)$ intersects $E_p$ only in the point $(0,0)$, hence $p_m(x,y)=x^m$. Since $f(C)$ and hence $\tilde{f}(C)$ are irreducible, we have that either $m=1$ and $r(x,y)=0$, or that $r(x,y)$ contains a term of the form $y^l$ for some $l\geq m+1$. Assume that the second assertion is true. In this case, the polynomial $s(x',y')$ contains a term of the form $y'^{l-m}$. We now blow up again the origin in the chart given by the coordinates $(x',y')$ and consider the strict transform of $\tilde{f}(C)$. If we iterate this process finitely many times we eventually obtain a strict transform that intersects the exceptional divisor not just in the origin. But this is not possible, since the lifts of $f$ induce a local isomorphism around all the exceptional divisors and restricts to the identity map on the exceptional divisors. We conclude that $f(C)$ is given by $\{x=0\}$. By choosing different coordinates, we obtain with the same reasoning that $f$ preserves each line in $U$ passing through $p$. We complete $U$ to {$\p^2$}  and observe that $f$ induces an automorphism on {$\p^2$}, which has to be the identity. 
\end{proof}

\begin{remark}\label{remark:Iso}
	Given two integers $r \geq 1$ {$d\geq2$}, let $\mathcal{T}_{d,r}$ denote the rooted tree whose root has degree $r$ and all of whose other vertices have degree $d+1$. If $r=n(d-1)+1$ for some integer $n \geq 1$, then the two groups $\AAut(\T_{d,r})$ and $\AAut(\T_{d})$ are isomorphic as topological groups (since there exists an almost isomorphism from $\mathcal{T}_{d,r}$ to $\mathcal{T}_{d}$).
	
	\medskip
	Hence, when working over the finite field $\F_q$, since the action is faithful, we obtain an embedding of $\Bir_{\F_q}(\p^2)$ into $\AAut(\FF\p^2(\F_q))\simeq\AAut(\T_{d,r})$, for $d=q+1$ and $r=(q+1)q+1$, and hence we can consider  $\Bir_{\F_q}(\p^2)$ as a subgroup of $\mathcal{N}_d$.
\end{remark}

\begin{remark}\label{rem:topology}
	Let us observe that a sequence of birational transformations $\{f_n\}$ in $\Bir_{\F_q}(\p^2)$ converges towards the identity with respect to this topology if there exists a sequence of regular projective surfaces $S_n$ above $\p^2$ of increasing height, such that, for $n$ large enough, $f_n$ induces a birational transformation on $S_n$ fixing all the rational points.
\end{remark}

It is a natural question to ask, which permutations of the rational points $S(\F_q)$ of a surface $S$ over $\F_q$ can be induced by elements in $\BBir(S(\F_q))$. In general, not all permutations of $S(\F_q)$ are induced by elements in $\BBir(S(\F_q))$. However, Cantat showed the following:

\begin{proposition}[{\cite[Theorem~2.1.]{Cantat-Bir_permut}}]\label{prop:cantat}
	Let $\F_q$ be a finite field with $q$ elements and assume that $q$ is odd or $q=2$. Then every permutation of $\p^2(\F_q)$ can be realized by an element in $\BBir_{\F_q}(\p^2)$. 
	
	If $q$ is even, then every even permutation of $\p^2(\F_q)$  can be realized by an element in $\BBir_{\F_q}(\p^2)$. 
\end{proposition}

Let us start with the following partial generalization of Proposition~\ref{prop:cantat}:

\begin{proposition}\label{prop:permutation}
	Let $X$ and $Y$ be rational projective surfaces over a finite field $\F_q$ with $q$ elements. Assume moreover that $X$ and $Y$ lie \emph{above} $\p^2$, i.e.\ that there exist birational morphisms $\pi\colon X\to \p^2$ and $\rho\colon Y\to \p^2$. 
	\begin{enumerate}
		\item If $|X(\F_q)|=|Y(\F_q)|$ then there exists a birational map $\varphi\colon X \dashrightarrow  Y$ that is bijective on rational points.
		\item 	Every even permutation of $X(\F_q)$ is induced by an element $f\in\BBir(X)$.
	\end{enumerate}
\end{proposition}

\begin{proof}
By considering only rational base-points we may assume that $\pi$ and $\rho$ are blow-ups in rational points. The \emph{height} of $X$ is the number of base-points of $\pi^{-1}$. Let us observe that $|X(\F_q)|=|Y(\F_q)|$ is equivalent to $X$ and $Y$ having the same height $h$. We will prove the proposition by induction:
	
	\begin{fact}\label{Fact_h1}
		Claim (1) and (2) hold for $h=0$ and Claim (2) holds for $h=1$.
	\end{fact}
	For $h=0$, Claim (1) is trivially true and Proposition~\ref{prop:cantat} states that Claim (2) is true as well. As a next step, we assume $h=1$ and we will show that Claim (2) holds. In this case, $\pi\colon X\to\p^2$ is the blow-up of a single rational point $a\in\p^2{(\F_q)}$ with exceptional divisor $E$. Let now $b\in\p^2$ be a point of degree $2$, in other words, over the algebraic closure, $b$ corresponds to two points $b_1, b_2$ that form a Galois orbit. Consider now the standard quadratic transformation $\sigma_{a,b}$ of $\p^2$ given by the linear system of conics passing through $a$ and $b$. In other words, over the algebraic closure of $\F_q$, $\sigma_{a,b}$ corresponds to the quadratic transformation given by blowing up $a, b_1,$ and $b_2$ and contracting the line defined over $\F_q$ passing through $b_1$ and $b_2$ as well as the two lines passing through $a$ and $b_i$ for $i=1,2$. We now observe that $\sigma_{a,b}\in\BBir(X)$ and that $\sigma_{a,b}$ maps the rational points on $E$ to the rational points on a line $L$ disjoint from $E$.
	\medskip

	Denote by $\BBir(X)_E\subset\BBir(X)$ the subgroup of elements preserving the set $E(\F_q)$ (and hence also $X(\F_q)\setminus E(\F_q)$). Observe that the elements in $\BBir_{\F_q}(\p^2)$ fixing $a$ lift to a subgroup of $\BBir(X)_E$. By Proposition~\ref{prop:cantat}, we therefore get that by restriction $\BBir(X)_E$ induces a surjective homomorphism $\rho\colon \BBir(X)_E\to \Sym(X(\F_q)\setminus E(\F_q))$. In particular, the image of $\rho$ contains all even permutations. Consider now the homomorphism $\eta\colon \BBir({X})_E\to \Sym(E(\F_q))$ given by restriction. Then $\rho(\ker \eta)\subset \Sym(X(\F_q)\setminus E(\F_q))$ is a normal subgroup. Note that $\rho({\ker }\eta)$ is non-trivial, since otherwise we would obtain a surjection $\BBir(X)_E/\ker \eta\to\BBir(X)_E/\ker\rho$, but no subgroup of $\Sym(E(\F_q))$ surjects to the image of $\rho$. Since $X(\F_q)\setminus E(\F_q)\geq 6$, the only non-trivial proper normal subgroup of $\Sym(X(\F_q)\setminus E(\F_q))$ is the subgroup of even permutations $\Sym^+(X(\F_q)\setminus E(\F_q))$, therefore, $\Sym^+(X(\F_q)\setminus E(\F_q))\subset\rho(\ker(\eta))$. In other words, we can realize every even permutation on $X(\F_q)\setminus E(\F_q)$ by an element in $\BBir(X)$ fixing $E(\F_q)$ pointwise. Together with $\sigma_{a,b}$ we can therefore realize every even permutation on $X(\F_q)$ that fixes $X(\F_q)\setminus E(\F_q)$ pointwise, and, in a second step, every even permutation of $X(\F_q)$ (by using the fact that the even permutations are generated by three-cycles and the fact that we have enough points in $X(\F_q)$ that are neither contained in $E(\F_q)$ nor in $\sigma_{a,b}(E(\F_q))$.).
	 So Claim (2) holds for $h=1$.
	
	\begin{fact}
		If Claim (1) and (2) hold for all $X$ and $Y$ of height $\leq h$, then Claim (1) also holds for all $X$ and $Y$ of height $h+1$.
	\end{fact}
	Assume that $X$ is obtained from $\p^2$ by blowing up rational points $p_1,\dots p_{h+1}$ and $Y$ is obtained by blowing up rational points $q_1,\dots, q_{h+1}$. Up to reordering the $p_i$ and $q_i$, we may assume that $X$ is the blow-up of $p_{h+1}$ on the surface $X'$ of height $h$ and $Y$ is the blow-up of $q_{h+1}$ on the surface $Y'$ of height $h$. Using (1) and (2) we know that there exists a birational map $\varphi'\colon X' \dashrightarrow  Y'$ bijective on rational points that maps $p_{h+1}$ to $q_{h+1}$. In particular, $\varphi'$ is a local isomorphism at $p_{h+1}$, so it lifts to a birational map $\varphi\colon X\dashrightarrow Y$ that is bijective on the rational points.
	
	\begin{fact}
		If Claim (2) holds for all $X$ of height $\leq h$ and Claim (1) holds for all surfaces $X$ and $Y$ of height $\leq h+1$, where $h\geq 1$, then Claim (2) holds for all surfaces $X$ of height $h+1$.
	\end{fact}

	Let $X$ be a surface of height $h+1$. By using Claim (1), we may assume that $X$ is the blow-up of two proper rational points $p$ and $s$ on a surface $X'$ of height $h-1$. Denote by $E_p$ and $E_s$ the exceptional divisors on $X$ corresponding to $p$ and $s$ respectively. Let $X_p$ and $X_s$ be the surface obtained from $X'$ by blowing up respectively $p$ and $s$. By using Claim (2), every even permutation on $X_p(\F_q)$ can be realized by an element in $\BBir(X_p)$ and every even permutation on $X_s(\F_q)$ can be realized by an element in $\BBir(X_s)$. Moreover, {since $X'(\F_q)$ contains more than four points,} there exists an element $\tau$ in $\BBir(X')$ that exchanges $p$ and $s$ and hence lifts to an element in $\BBir(X)$ exchanging the sets $E_p(\F_q)$ and $E_s(\F_q)$ with each other. With the same reasoning as in the proof of Fact \ref{Fact_h1} (where $\tau$ takes the role of $\sigma_{a,b}$), this implies that $\BBir(X)$ induces all even permutations on $X(\F_q)$.
\end{proof}
Let $f\in\AAut(\FF\p^2(\F_q))$ be represented by $(\varphi, F, F)$. Note that Proposition~\ref{prop:permutation} implies in particular that there exists an element $g\in\Bir_{\F_q}(\p^2)$ that can be represented by $(\psi, F, F)$ such that $\varphi$ and $\psi$ induce the same permutation on {the roots of $ F^C$}. Indeed, let $F'$ be the rooted forest obtained by adding all the descendants to {the roots of $ F^C$}. Proposition~\ref{prop:permutation} implies that every even permutation of $X_{F'}(\F_q)$ can be realized by some element in $\BBir(X_{F'})$. In particular, every even permutation on the roots of $F'^C$ is induced by an element $(\psi, F',F')$ given by an element from $\BBir(X_{F'})$. This implies that every permutation (even if it is not even) on the roots of $F^C$ is induced by some $(\psi, F, F)$ given by an element in $\BBir(X_{F'})$ and hence by an element $g\in\Bir_{\F_q}(\p^2)$. (This $g$ however, is not always conjugate to an element in $\BBir(X_{F})$, by Theorem~\ref{thm:mainparity}).

\begin{lemma}
Consider $\Bir_{\F_q}(\p^2)$ over $\F_q$ as a subgroup of $\AAut(\FF({\F_q}))$. Then the intersection $\Bir_{\F_q}(\p^2)\cap \Aut(\FF\p^2({\F_q}))$ is dense in $\Aut(\FF\p^2({\F_q}))$.
\end{lemma}

\begin{proof}
	We need to show that for every {rooted} forest automorphism $f\in \Aut(\FF\p^2(\F_q))$ and for every $\epsilon>0$, there exists a $g\in\Bir_{\F_q}(\p^2){\cap\Aut(\FF\p^2(\F_q))}$ such that $d(f,g)<\epsilon$. Let $F_r\subset \FF\p^2(\F_q)$ be the admissible subforest containing all the vertices of distance at most $r$ from the roots. Then $f(F_r)=F_r$. By Proposition~\ref{prop:permutation}, every even permutation of the {roots of $F_r^C$} is induced by an element of $\Bir_{\F_q}(\p^2)$. This implies, as explained above, that every permutation {of the roots of $F_{r}^C$} is induced by an element of $\Bir_{\F_q}(\p^2)$. In particular, there exists a $g\in\Bir_{\F_q}(\p^2)$ such that $g$ induces the same permutation on the roots of $F_r^C$ as $f$. This implies, by the definition of the topology on $\Aut(\FF\p^2(\F_q))$, that $d(g,f)\leq e^{-r}$. 
\end{proof}

The following result together with Lemma~\ref{lem:faithful} and Remark~\ref{remark:Iso} proves Theorem~\ref{thm:MainNeretinCremona}.

\begin{proposition}
	When working over the finite field $\F_q$, the subgroup $\Bir_{\F_q}(\p^2)\subset \AAut(\FF\p^2(\F_q))$ is dense.
\end{proposition}

\begin{proof}
	We will show that every coset of $\Aut(\FF\p^2(\F_q))$ in $\AAut(\FF\p^2(\F_q))$ can be represented by an element in $\Bir_{\F_q}(\p^2)$. Let $f\in\AAut(\FF\p^2(\F_q))$ be an element represented by {$(\varphi, F, R)$}. After possibly enlarging $F$, we may assume that $F=F_r$ for some $r$, where $F_r\subset \FF\p^2(\F_q)$ is the admissible subforest containing all the vertices of distance at most $r$ from the roots.
	\medskip
	
Consider two blow-ups $\pi_F\colon S_F\to \p^2$ and $\pi_R\colon S_R\to \p^2$ corresponding to $F$ and $R$ respectively. Since $F{^C}$ and $R{^C}$ have the same number of roots, $|S_F(\F_q)|=|S_R(\F_q)|$, hence by Proposition~\ref{prop:permutation}, there exists a birational map $g\colon S_F\dashrightarrow S_R$ that is bijective on rational points. The element $\pi_R g\pi_F^{-1}\in\Bir_{\F_q}(\p^2)$ therefore induces an almost automorphism of the form $(\psi, F,R)$ for a suitable $\psi$.   By Proposition~\ref{prop:permutation}, there exists an element in $\Bir_{\F_q}(\p^2)$ that induces the same map on the {roots of $R^C$} as the map $(\psi, F,R)(\varphi, F,R)^{-1}=(\psi\circ\varphi{^{-1}}, R,R)$. Hence, up to composing with another element from $\Bir_{\F_q}(\p^2)$, we may assume that  $(\psi, F,R)(\varphi, F,R)^{-1}$ is in $\Aut(\FF\p^2(\F_q))$.
\end{proof}

\subsection{Parity}\label{Subsection_parity}
The goal of this section is to prove that if $q=2^n$, where $n\geq 2$, and $S$ is a regular projective rational surface over $\F_q$, then the permutation on $S(\F_q)$ induced by an element $f\in\BBir(S)$ is always even. The result will essentially follow from our point of view of looking at $\Bir_{\F_q}(\p^2)$ over $\F_q$ as the group of almost automorphisms of the blow-up {forest}, and the following three results (Lemma~\ref{lem:pgl2}, Theorem~\ref{thm:finiteorder}, and Theorem~\ref{thm:lamyschneider}).
The first lemma is not hard to prove (see for instance \cite[Theorem~2.1]{Zimmermann-bir_permutations}). Note that the statement does no longer hold if $q=2$ or if $q$ is odd ({see Proposition \ref{prop:cantat}}). 

\begin{lemma}\label{lem:pgl2}
	Let $\F_q$ be a finite field of even characteristic with $q\geq 4$ elements. Then the permutations on $\p^1(\F_q)$ induced by $\PGL_2(\F_q)$ are even.
\end{lemma}

The following result is more difficult to show:

\begin{theorem}[{\cite{Zimmermann-bir_permutations}}]\label{thm:finiteorder}
Let $q=2^n\geq 4$ and let $f\in\Bir_{\F_q}(\p^2)$ over $\F_q$ be of finite order. Then there exists a regular projective surface $S$ over $\F_q$ such that $f$ is conjugate by a birational map to an automorphism of $S$ and the permutation on $S(\F_q)$ induced by $f$ is even. 
\end{theorem}

Theorem~\ref{thm:finiteorder} is proven in {\cite[Corollary~3.17]{Zimmermann-bir_permutations}}, which states that every automorphism of a del Pezzo surface $S$ induces an even permutation on $S(\F_q)$, {\cite[Theorem~3.18]{Zimmermann-bir_permutations}}, which states that every automorphism of a conic bundle $\mathcal{C}$ over $\p^1$ induces an even permutation on $\mathcal{C}(\F_q)$, and {\cite[Lemma~3.13]{Zimmermann-bir_permutations}}, which states that  {over $\F_q$} every element of finite order in $\Bir_{\F_q}(\p^2)$ is conjugate to an automorphism of a del Pezzo surface or an automorphism of a conic bundle over $\p^1$.
\medskip 

The last crucial ingredient for our proof is the following very recent result by Lamy and Schneider:

\begin{theorem}[\cite{lamy2021generating}]\label{thm:lamyschneider}
	The Cremona group $\Bir_{\F_q}(\p^2)$ over a finite field $\F_q$ is generated by involutions.
\end{theorem}

In the remaining part of this section we assume that $q$ is even and that $q\neq 2$. 

\medskip

Given a rational regular projective surface $S$, let us fix an embedding of our rational blow-up {forest $\FF S(\F_q)$} into the plane $\R^2$. We may assume that the roots lie on the line $y=0$ and that all the vertices on level $n$ lie on the line $y=n$. Let {$F\subset \FF S(\F_q)$ be a rooted subforest}. Then the set of {roots of $F^C$} inherits the ordering from left to right. If $X$ is a surface above $S$, then this embedding induces in particular an order on the points $X(\F_q)$. In that way, we can define a notion of parity for every birational map $f\colon X\dashrightarrow Y$ bijective on rational points, where $X$ and $Y$ are surfaces above $S$: We identify both sets $X(\F_q)$ and $Y(\F_q)$ with $\{1,\dots, n\}$ by the left to right order and we call the map $f$ \emph{even}, if the corresponding permutation of $\{1,\dots, n\}$ is even. Let us make the following observation:

\begin{lemma}
	Given a rational regular projective surface $S$, we can choose the embedding of $\FF S(\F_q)$ into $\R^2$ in such a way that the following property is satisfied: 
	
	\begin{enumerate}[(A)]
		\item For all regular projective surfaces $X$ and $Y$ above $S$, for all birational maps $f\colon X\dashrightarrow Y$ bijective on rational points, and for all $p\in X(\F_q)$ such that $f$ is a local isomorphism around $p$ the following holds: Let $X'$ be the blow-up of $X$ in $p$ and $Y'$ the blow-up of $Y$ in $f(p)$, with exceptional divisors $E_p$ and $E_{f(p)}$. We identify $E_p(\F_q)$ and $E_{f(p)}(\F_q)$ with the set $\{1,\dots, q+1\}$ by virtue of the left to right order. Then the bijection between $E_p(\F_q)$ and $E_{f(p)}(\F_q)$ induced by $f$ is even.
	\end{enumerate}
\end{lemma}

\begin{proof}

Choosing an embedding of $\FF S(\F_q)$ into $\R^2$ amounts to choosing a left to right order on the points on each height that is compatible with the forest structure. We define these orders inductively. First, choose any order on the points $S(\F_q)$. Next, choose a point $p'\in S(\F_q)$ and let $E_{p'}$ be the exceptional divisor obtained by blowing up $p'$ and fix an order on the points $E_{p'}(\F_q)$. Assume now that we have defined an embedding of the blow-up forest up to height $n$. Let now $p''=(p'', S'', \pi)$ be a point in the bubble space of height $n$. We now need to define an order on the rational points $E_{p''}(\F_q)$, where $E_{p''}$ is the exceptional divisor obtained by blowing up $S''$ in $p''$. For this, we choose a representative $(p'', S'', \pi)$. 
There exists a birational map $\varphi_{p''}\colon S\dashrightarrow S''$ that is a local isomorphism around $p'$ such that $\varphi_{p''}(p')=p''$. We choose such a local isomorphism. Being a local isomorphism, $\varphi_{p''}$ induces a bijective map $E_{p'}(\F_q)\to E_{p''}(\F_q)$. We now equip $E_{p''}(\F_q)$ with the order induced by $E_{p'}(\F_q)$ through this bijection. We continue in that way until we have an embedding of the vertices up to height $n$ and then continue inductively.
 Of course, this order on the points $E_{p''}(\F_q)$ depends on the choice of $\varphi_{p''}$.
 \medskip 
 
 Let us now show that the embedding we obtain in that way (no matter which choices we make) satisfies the desired property (A). Let $X, Y, f,$ and $p$ be as stated and assume that we identify the rational blow-up forests $\mathcal{F} X(\F_q)$ and $\mathcal{F} Y(\F_q)$ to subforests of $\FF S(\F_q)$ through the rational morphisms $\pi\colon X\to S$ and $\rho\colon Y\to S$. So if we consider $p$ and $f(p)$ as elements of $\B S(\F_q)$, we can represent them by $(p, X, \pi)$ and $(f(p), Y, \rho)$ respectively.  Let $(r, X', \pi')$ and $(r',Y', \rho')$ be the representatives of $p$ and $f(p)$ that we chose when defining the order on $E_p$ and $E_{f(p)}$, and let $\varphi_r\colon S\dashrightarrow X'$, $\varphi_{r'}\colon S\dashrightarrow Y'$ be the chosen local isomorphisms around $p'$. We can now observe that the birational transformation $\Psi:=(\varphi_{r'}^{-1}\rho'^{-1}\rho)f(\pi^{-1}\pi'\varphi_r)\colon S\dashrightarrow S$ is a local isomorphism around $p'$ and $\Psi(p')=p'$. Therefore, $\Psi$ induces a morphism from $E_{p'}$ to itself which is given by an element from $\PGL_2$ and hence, by Lemma~\ref{lem:pgl2}, induces an even permutation on $E_{p'}(\F_q)$. Since, by construction, $\varphi_r$ and $\varphi_{r'}$ induce order preserving maps from $E_{p'}(\F_q)$ to $E_p(\F_q)$ and $E_{f(p)}(\F_q)$ respectively, this implies that the bijection $E_p(\F_q)\to E_{f(p)}(\F_q)$ induced by $f$ is even.
\end{proof}

From now on, we will always assume that $\FF S(\F_q)$ is embedded in $\R^2$ in such a way that property (A) is satisfied.

\begin{lemma}\label{lem:blowupparity}
	Let $X$ and $Y$ be regular projective surfaces above a rational regular projective surface $S$, let $f\colon X\dashrightarrow Y$ be a birational map bijective on rational points, and let $\epsilon$ be the parity of the induced map $f\colon X(\F_q)\to Y(\F_q)$, if we identify $X(\F_q)$ and $Y(\F_q)$ with $\{1,\dots, n\}$. Then for every $p\in X(\F_q)$ the induced bijective map $\tilde{f}\colon X_p(\F_q)\to Y_{f(p)}(\F_q)$ has also parity $\epsilon$, where $X_p$ and $Y_{f(p)}$ are the surfaces obtained by blowing up $p$ {and} $f(p)$, respectively. 
\end{lemma}

\begin{proof}
	Assume that $p$ is on position $s$ and $p'=f(p)$ is on position $r$. 
	The rational points $E_p(\F_q)$ on the exceptional divisor of the blow-up of $p$ correspond to the numbers $\{s,\dots, s+q\}$ under the identification of $X_p(\F_q)$ with $\{1,\dots, n+q\}$ and the rational points $E_{p'}(\F_q)$ on the exceptional divisor of the blow-up of $p'$ correspond to the numbers $\{r,\dots, r+q\}$.  The situation is illustrated in Figures~\ref{fig:X} and \ref{fig:Y}. 
	
%
\begin{figure}
		\begin{subfigure}[a]{\textwidth}
			\begin{center}
\begin{tikzpicture}
	\node (1) at (-0.5,0) {$1$};
	\node (2) at (0.5,0) {$\dots$};
	\node (3) at (1.5,0) {$s-1$};
	\node[draw,circle] (4) at (3,0) {$p$};
	\node (5) at (4.5,0) {$s+q+1$};
	\node (6) at (6,0) {$\dots$};
	\node (7) at (7.5,0) {$r+q+1$};
	\node (8) at (9,0) {$\dots$};
	\node (9) at (10.5,0) {$n+q$};
	\node (10) at (1,-1) {$s$};
	\node (11) at (2.5,-1) {$s+1$};
	\node (12) at (3.7,-1) {$\dots$};
	\node (13) at (5,-1) {$s+q$};
	\draw (4) -- (10);
	\draw (4) -- (11);
	\draw (4) -- (12);
	\draw (4) -- (13);
\end{tikzpicture}			\end{center}
\caption{The rational points of $X_p$ in their left to right order.}\label{fig:X}
	\end{subfigure}

	\begin{subfigure}[b]{\textwidth}
		\vspace{10mm}
				\begin{center}
		\begin{tikzpicture}
			\node (1) at (-0.5,0) {$1$};
			\node (2) at (0.5,0) {$\dots$};
			\node (3) at (1.5,0) {$s$};
			\node (4) at (2.5,0) {$\dots$};
			\node (5) at (3.5,0) {$r-1$};
			\node[draw,circle] (6) at (5,0) {$p'$};
			\node (7) at (6.5,0) {$r+q+1$};
			\node (8) at (8,0) {$\dots$};
			\node (9) at (9.5,0) {$n+q$};
			\node (10) at (3,-1) {$r$};
			\node (11) at (4.5,-1) {$r+1$};
			\node (12) at (5.7,-1) {$\dots$};
			\node (13) at (7,-1) {$r+q$};
			\draw (6) -- (10);
			\draw (6) -- (11);
			\draw (6) -- (12);
			\draw (6) -- (13);
		\end{tikzpicture}
\end{center}
	\caption{The rational points of $Y_{f(p)}$ in their left to right order.}\label{fig:Y}
			\vspace{5mm}
			\end{subfigure}
			\caption{The order on the rational points on the blow-ups.}
	\end{figure}
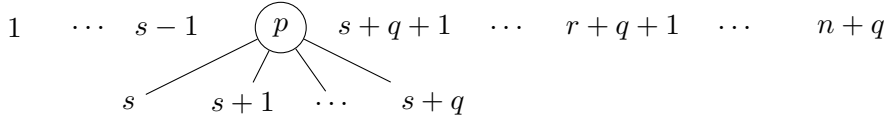
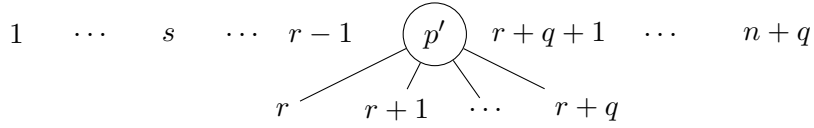

	 Let $\tau$ be the permutation of $\{1,\dots, n\}$ induced by $f$ and let $\tilde{\tau}$ be the permutation on $\{1,\dots, n+q\}$ induced by $\tilde{f}$. Hence we have the following commutative diagram:

	\begin{center}
		\begin{tikzcd}
		\{1,\dots, (s,\dots, s+q), \dots, n+q\} \arrow[rr, "\tilde{\tau}"] \arrow[d, "\alpha"] &  & \{1,\dots,(r,\dots, r+q),\dots, n+q\} \arrow[d, "\beta"] \\
		\{1,\dots, s,\dots, n\} \arrow[rr, "\tau"]                    &  & \{1,\dots, r, \dots, n\},                   
		\end{tikzcd}
	\end{center}
	
	where $\alpha$ maps $\{s,\dots, s+q\}$ to $s$ and is bijective and order preserving on $\{1,\dots, s-1, s+q+1,\dots, n+q\}$, and $\beta$ maps $\{r,\dots, r+q\}$ to $r$ and is bijective and order preserving on $\{1,\dots, r-1, r+q+1,\dots, n+q\}$. Since the embedding of $\FF S(\F_q)$ satisfies property (A), we may assume - up to composing $\tilde{\tau}$ by an even permutation - that the restriction of $\tilde{\tau}$ to $\{s,\dots, s+q\}$ preserves the order.
	\medskip
	
	Let $\sigma$ be the permutation of $\{1,\dots, n\}$ that maps $r$ to $s$ and that is order preserving on $\{1,\dots,r,\dots n\}$ and let $\tilde{\sigma}$ be the permutation of $\{1,\dots, n+q\}$ such that $\tilde{\sigma}(\{r,\dots, r+q\})=\{s,\dots, s+q\}$ and such that the restriction to $\{1,\dots, r-1, r+1,\dots, n+q\}$ as well as to $\{r,\dots, r+q\}$ is order preserving. Then $\sigma\tau$ fixes $s$ and $\tilde{\sigma}\tilde{\tau}$ fixes $\{s,\dots, s+q\}$ pointwise. We obtain the following commutative diagram:
	
		\begin{center}
		\begin{tikzcd}
		\{1,\dots, (s,\dots, s+q), \dots, n+q\} \arrow[rr, "\tilde{\sigma}\tilde{\tau}"] \arrow[d, "\alpha"] &  & \{1,\dots, (s,\dots, s+q), \dots, n+q\} \arrow[d, "\alpha"] \\
		\{1,\dots, s,\dots, n\} \arrow[rr, "\sigma\tau"]                    &  & \{1,\dots, s, \dots, n\}.                 
		\end{tikzcd}
	\end{center}

	Let us now observe that $\tilde{\sigma}\tilde{\tau}$ and $\sigma\tau$ have the same parity. Moreover, the parity of $\sigma$ is $(-1)^{r-s}$ and the parity of $\tilde{\sigma}$ is $(-1)^{(r-s)(q+1)}$. By assumption, $q$ is even and hence $\sigma$ and $\tilde{\sigma}$ have the same parity. We obtain that $\tau$ and $\tilde{\tau}$ have the same parity, which proves the lemma. 
\end{proof}

To an almost automorphism of $\FF S(\F_q)$ represented by a triple $(\psi, F,F')$ we would like to associate a notion of parity by saying that $\psi$ is \emph{even}, if the induced bijection on the {roots of the complements $\lambda( F^C)\to \lambda({F'}^C)$} is even for all $F$ and $F'$ large enough, with respect to the above defined ordering, and \emph{odd}, if the the bijection is odd. If $q$ is odd or $q=2$, this notion is not well defined, since it depends on the choice of the representative $(\psi, F,F')$. However, the next lemma shows with the help of Lemma~\ref{lem:pgl2} that the notion is well defined for almost automorphisms induced by birational maps $\Bir_{\F_q}(\p^2)$ over $\F_q$, if $q$ is even and $q\neq 2$.

\begin{lemma}\label{lem:parity}
	Assume that $q$ is even and $q\neq 2$ and let $S$ be a rational regular projective surface over $\F_q$. Then for all triples $(\psi, F, F')$ representing the almost automorphism on $\FF S(\F_q)$ given by an $f\in\Bir(S)$, such that $F$ contains the vertices corresponding to the rational base-points of $f$ and $F'$ contains the vertices corresponding to the rational base-points of $f^{-1}$, the parity of the induced map $\lambda( F^C)\to \lambda(F'^C)$ is the same. 
	
	This will be called the \emph{parity of $f$}.
\end{lemma}

\begin{proof}
	Let $(\psi, F, F')$ be a triple representing $f$ and denote by $\pi_F\colon S_F \to S$ and $\pi_{F'}\colon S_{F'}\to S$ the blow-ups corresponding to $F$ and $F'$ respectively. Since $F$ and $F'$ contain the base-points of $f$ and $f^{-1}$ respectively, we can assume that $\psi$ corresponds to the bijection $S_F(\F_q)\to S_{F'}(\F_q)$ induced by the birational map $\pi_{F'}^{-1}f\pi_F$ that is bijective on the rational points. Let $(\tilde{\psi}, \tilde{F}, \tilde{F}')$ be another triple representing $f$ such that $F\subset\tilde{F}$ and $F'\subset \tilde{F}'$. 
	We will now show that the parity of the induced map on $\lambda( \tilde{F}^C)\to\lambda(\tilde{F}'^C)$ is the same as the parity of the induced map $\lambda( F^C)\to \lambda(F'^C)$. We proceed inductively on the number of interior vertices of $F$. To do so, assume that $\tilde{F}$ has one interior vertex more than $F$. Again, we denote by $\pi_{\tilde{F}}\colon S_{\tilde{F}}\to S$ and $\pi_{\tilde{F}'}\colon S_{\tilde{F}'}\to S$ the blow-ups corresponding to $\tilde{F}$ and $\tilde{F}'$. As $\tilde{F}$ has one interior vertex more than $F$, there exists a point $p\in S_F$ such that $S_{\tilde{F}}$ is obtained from $S_F$ by blowing up $p$. Similarly, $S_{\tilde{F}'}$ is obtained from $S_{F'}$ by blowing up the point $\pi_{F'}^{-1}f\psi_F(p)$. It follows now from Lemma~\ref{lem:blowupparity} that the map $S_{\tilde{F}}(\F_q)\to S_{\tilde{F}'}(\F_q)$ induced by $\pi_{\tilde{F}'}^{-1}f\pi_{\tilde{F}}$ has the same parity as the map $S_F(\F_q)\to S_{F'}(\F_q)$ induced by $\pi_{F'}^{-1}f\pi_F$ and consequently that the map $$\lambda( \tilde{F}^C)\to \lambda(\tilde{F}'^C)$$ induced by $\tilde{\psi}$ has the same parity as the map $\lambda( F^C)\to \lambda(F'^C)$ induced by $\psi$. This completes the proof.
\end{proof}

\begin{lemma}\label{lem:ParBBir}
	Let $f\in\Bir_{\F_q}(\p^2)$ be a birational transformation over $\F_q$ that is conjugate by a birational map to an element $\tilde{f}\in\BBir(S)$, for some regular projective surface $S$. Then the parity of $f$ coincides with the parity of the induced map $S(\F_q)\to S(\F_q)$. 
\end{lemma}

\begin{proof}
Let $S'$ be the surface obtained by iteratively blowing up all the rational points on $S(\F_q)$ so that it dominates $\p^2$. Then $f$ is conjugate to an element $f'\in\BBir(S')$, since in particular, the rational base-points of $f$ and $f^{-1}$ have been blown up to obtain $S'$. Consequently, by Lemma \ref{lem:parity}, the parity of $f$ is given by the parity of the map $f'$.  By applying Lemma~\ref{lem:parity}, we see that the parity of $f'$ is the same as the parity of the map $S(\F_q)\to S(\F_q)$ induced by $\tilde{f}$.
\end{proof}

\begin{lemma}\label{lem:ParComp}
Let $f,g\in\Bir_{\F_q}(\p^2)$ over $\F_q$, where $q$ is even and $q\neq 2$. Then the parity of $fg$ is the product of the parities of $f$ and $g$.
\end{lemma}

\begin{proof}
	
	This follows directly from the definition of the parity. Let $(\varphi, R, F)$ be the almost automorphism of $\FF\p^2(\F_q)$ induced by {$f$}and $(\psi, U,R)$ the one induced by {$g$}. Then, for all large enough $R$, the map {$\lambda(R^C)\to \lambda(F^C)$} induced by $\varphi$ has the parity of $g$ and the map  {$\lambda(U^C)\to\lambda(R^C)$} induced by $\psi$ has the parity of $f$. Therefore, the parity of the almost automorphism $(\varphi\psi, {U},{F})$ induced by $fg$ is the product of the parities of $f$ and~$g$.
\end{proof}

We have now all the ingredients to show Theorem~\ref{thm:mainparity}:

\begin{proof}[Proof of Theorem~\ref{thm:mainparity}]
 Firstly, by Theorem~\ref{thm:finiteorder}, an involution on $\Bir_{\F_q}(\p^2)$ is always conjugate to an element in $\BBir(S)$ for some regular projective surface $S$ and the corresponding permutation on $S(\F_q)$ is even. Secondly, by Lemma~\ref{lem:ParBBir} this implies that involutions are even. Finally, by Theorem~\ref{thm:lamyschneider} $\Bir_{\F_q}(\p^2)$ is generated by involutions, so we conclude with Lemma~\ref{lem:ParComp} that all birational transformations in $\Bir_{\F_q}(\p^2)$ are even.
\medskip

By applying again Lemma~\ref{lem:ParBBir}, we obtain that for a rational regular projective surface $S$, all elements in $\BBir(S)$ induce even permutations on $S(\F_q)$.
\end{proof}

\begin{remark}\label{rem:AutD}
	{In \cite{MR2806497} and \cite{MR3612334}, the subgroup $\AAut_D(\mathcal{T}_{d,n})\subset \AAut(\mathcal{T}_{d,n})$ are defined and studied for a subgroup $D\subset S_d$ in the following way: Fix an embedding of $\mathcal{T}_{d,n}$ into the plane. Denote by $W(D)$ the subgroup of automorphisms of the $d$-regular rooted tree $\mathcal{\T}_d$ that act on the vertices on the first level  by $D_1:=D$ and by $D_{n+1}:=D\wr D_n$ on the vertices on the $n$-th level for $n>1$. A quasi-automorphism $f\in \AAut(\mathcal{T}_{d,n})$ is now defined to be contained in $\AAut_D(\mathcal{T}_{d,n})$ if $f$ can be represented by a triple $(\varphi, T, T')$ such that the forest isomorphism $\varphi\colon T^C\to T'^C$ belongs to $W(D)$ on each connected component after identifying each connected  component of the forest with $\T_d$ through its embedding into the plane. }
	If $D=\mathcal{A}_d$ is the group of even permutations, one can show similarly as above that there exists a notion of parity for elements in $\AAut_D(\mathcal{T}_{d,n})$.
\end{remark}

\section{Cubulation of Neretin groups}\label{Section_cubulation_Neretin}
{In this section, we assume the reader familiar with the notion of CAT(0) cube complexes, that are simply connected (and connected) non-positively curved cube complexes see for instance \cite{sageev_lecturenotes}.}

\subsection{Main construction}\label{section:MainConstruction}

Let $\mathcal{T}$ be a locally finite rooted tree and let $\mathcal{N}:= \mathrm{AAut}(\mathcal{T})$ denote its almost automorphism group. Recall that a subtree $A \subset \mathcal{T}$ is {admissible} if it contains the root and is finite, {and we denote by $\lambda(A^C)$ the set of the roots of the complementary forest $A^C$.} Observe that $\lambda(A^C)$ can also be described as the set of the children of $A$. 

\medskip
In the sequel, we will often refer to a forest isomorphism $R^C \to S^C$ as given by an element $\varphi \in \mathcal{N}$. Formally, this is not correct since $\varphi$ is a class and not a single transformation. By this abuse of language, we mean that there exists a forest isomorphism $ R^C \to  S^C$ defining the same element as $\varphi$ in $\mathcal{N}$.

\begin{definition}
Let $\mathscr{C}=\mathscr{C}(\mathcal{T})$ denote the cube complex
\begin{itemize}
	\item whose vertices are the classes $[A,\varphi]$ of pairs $(A,\varphi)$, where $A \subset \mathcal{T}$ is non-empty admissible subtree and $\varphi \in \mathcal{N}$, up to the equivalence: $(R,\mu) \sim (S,\nu)$ if $\mu^{-1}\nu$ is a forest isomorphism $ S^C \to  R^C$ that sends $\lambda(S^C)$ to $\lambda(R^C)$;
	\item whose edges connect any two vertices of the form $[A,\varphi]$ and $[A \cup b, \varphi]$ with $b \in \lambda(A^C)$, where $A \cup b$ denotes the subtree given by $A$ and the edge connecting $b$ to $A$;
	\item and whose $k$-cubes are spanned by subgraphs of the form $$\left\{ \left[ A \cup \bigcup_{i \in I} b_i , \varphi \right] \mid I \subset \{1, \ldots, k\} \right\}$$ where $b_1, \ldots, b_k$ are pairwise distinct vertices in $\lambda( A^C)$.
\end{itemize}
\end{definition}
The complex $\mathscr{C}$ is naturally endowed with a \emph{height function} $\mathfrak{h} : \mathscr{C}^{(0)} \to \mathbb{N}$. Namely, for every vertex $[A,\varphi]\in \mathscr{C}$, the height $\mathfrak{h}([A,\varphi])$ is the cardinality of $\lambda(A^C)$; observe that this number does not depend on the representative we choose. 

\medskip
It is straightforward that 
$$\psi \cdot [A,\varphi] := [A,\psi \circ \varphi], \ [A,\varphi]\in \mathscr{C}, \ \psi \in \mathcal{N}$$
defines an action of $\mathcal{N}$ on $\mathscr{C}$ by automorphisms preserving the height function.

\medskip
The rest of the section is dedicated to the proof of Theorem~\ref{thm:IntroCubulation}. In fact, we are going to prove a slightly more general result: instead of focusing on a regular tree $\mathcal{T}_d$, we consider a tree $\mathcal{T}_{d,r}$ whose root has $r$ children and all of whose other vertices have $d$ children.

\begin{theorem}\label{thm:MainNeretin}
If $\mathcal{T}=\mathcal{T}_{d,r}$ for some $d \geq 2$ and $r \leq d$, then the complex $\mathscr{C}$ is a locally finite CAT(0) cube complex on which the Neretin group $\mathcal{N}$ acts with compact-open stabilisers\footnote{Observe that the condition $r \leq d$ is not restrictive: indeed, if $r\geq d$, then the trees $\mathcal{T}_{d,r}$  and  $\mathcal{T}_{d,r-d+1}$are almost isomorphic, so the corresponding almost automorphism groups are isomorphic.} Moreover, every vertex-stabiliser equals the automorphism group of some cofinite rooted subforest and, conversely, for each cofinite  rooted subforest, there exists a vertex-stabiliser that is its automorphism group.
\end{theorem}

\noindent
Even though our main statement only deals with regular trees, we emphasize that several of our intermediate results hold for arbitrary locally finite trees. More precisely, Propositions~\ref{prop:LocallyFinite} and~\ref{prop:OneConnected} show that our cube complex $\mathscr{C}$ is always locally finite and simply connected (see also Remark~\ref{remark:CCinGeneral}).

\paragraph{Local finiteness.} {Let $\mathcal{T}$ be a locally finite rooted tree. }We begin by proving the following observation:

\begin{proposition}\label{prop:LocallyFinite}
The cube complex $\mathscr{C}$ is locally finite.
\end{proposition}

The next lemma, which characterises how to pass from a vertex to one of its neighbours with higher height, will be useful in the sequel.

\begin{lemma}\label{lem:GoingUp}
Let $x,y \in \mathscr{C}$ be two adjacent vertices such that $\mathfrak{h}(y) > \mathfrak{h}(x)$. If $(R,\varphi)$ is a representative of $x$, then there exists a vertex $b \in \lambda(R^C)$ such that $y=[R \cup b, \varphi]$. 
\end{lemma}

\begin{proof}
By construction of $\mathscr{C}$, we can write $x$ and $y$ respectively as $[A, \psi]$ and $[A \cup c, \psi]$. From the equality $[A,\psi]=[R,\varphi]$, we know that $\varphi^{-1}\psi$ defines a forest isomorphism $A^C \to R^C$ sending $\lambda(A^C)$ to $\lambda(R^C)$. Set $b:= \varphi^{-1}\psi(c) \in \lambda(R^C)$. Because $\varphi^{-1}\psi$ defines a forest isomorphism $(A \cup c)^C \to (R \cup b)^C$, we conclude as desired that $y=[A \cup c, \psi]=[R \cup b, \varphi]$. 
\end{proof}

\begin{proof}[Proof of Proposition~\ref{prop:LocallyFinite}.]
Let $x \in \mathscr{C}$ be a vertex. According to Lemma~\ref{lem:GoingUp}, $x$ has exactly $\mathfrak{h}(x)$ neighbours of higher height. So, in order to conclude, it is sufficient to bound the size of any collection $\{y_i \mid i \in I\}$ of neighbours of lower height. For every $i \in I$, fix a representative $(A_i,\psi_i)$ of $y_i$. We know from Lemma~\ref{lem:GoingUp} that there exists some $a_i \in \lambda(A^C)_i$ such that $(A_i \cup a_i, \psi_i)$ represents $x$. From the equality $[R,\varphi]=[A_i \cup a_i, \psi_i]$, we know that {$\varphi^{-1} \psi_i$} induces a forest isomorphism $(A_i \cup a_i)^C \to R^C$ sending $\lambda((A_i \cup a_i)^C)$ to $\lambda(R^C)$; let $c_i \subset \lambda(R^C)$ denote the images of the children of $a_i$ under $\varphi \psi_i^{-1}$. If $|I|> 2^{\# \lambda(R^C)}$, then there exist two distinct indices $i,j \in I$ such that $c_i=c_j$. Thus, {$(\varphi^{-1} \psi_j)^{-1} \circ \varphi^{-1} \psi_i = \psi_j^{-1} \psi_i$} induces a forest isomorphism $(A_i \cup a_i)^C \to (A_j \cup a_j)^C$ sending $\lambda((A_i \cup a_i)^C)$ to $\lambda((A_j \cup a_j)^C)$ and the children of $a_i$ to the children of $a_j$, which implies that $\psi_j \psi_i^{-1}$ induces a forest isomorphism $A^C_i \to A^C_j$ sending $\lambda(A^C)_i$ to $\lambda(A^C)_j$, hence 
$y_i= [A_i, \psi_i]=[A_j, \psi_j]=y_j.$
Therefore, we have proved that $x$ has only finitely many neighbours of lower height, as desired. 
\end{proof}

\paragraph{Simple connectivity.} The next step towards nonpositive curvature is to show that our cube complex is simply connected {when $\mathcal{T}$ is a locally finite rooted tree}. The subsection is dedicated to the proof of this assertion.

\begin{proposition}\label{prop:OneConnected}
The cube complex $\mathscr{C}$ is connected and simply connected.
\end{proposition}

\begin{proof}
Given two vertices $x,y \in \mathscr{C}$, we say that $y$ \emph{dominates} $x$ if there exists an \emph{increasing path} from $x$ to $y$ in $\mathscr{C}$, i.e.\ a path along which the next vertex has higher height than the previous one. We begin by observing:

\begin{claim}\label{claim:Dominate}
For every finite set of vertices $S \subset \mathscr{C}$, there exists a vertex $z \in \mathscr{C}$ that dominates all the vertices in $S$.
\end{claim}

\noindent
Let $(R_1, \varphi_1), \ldots, (R_k,\varphi_k)$ be representatives of the vertices in $S$. For every $1 \leq i \leq k$, fix two admissible subtrees $U_i,V_i \subset \mathcal{T}$ such that $\varphi_i$ defines a forest isomorphism $U_i^C \to V_i^C$ sending $\lambda(U_i^C)$ to $\lambda(V_i^C)$; without loss of generality, we can assume that $R_i \subset U_i$. Now, fix a subtree $A \subset \mathcal{T}$ that contains all the $V_i$. We claim that $[A,\mathrm{id}]$ dominates all the vertices in $S$.

\medskip \noindent
Indeed, given an index $1 \leq i \leq k$, there clearly exists an increasing path from $[R_i,\varphi_i]$ to $[U_i, \varphi_i]$ since $U_i$ contains $R_i$. But $[U_i,\varphi_i]= [V_i, \mathrm{id}]$, by definition of $U_i$ and $V_i$. And there clearly exists an increasing path from $[V_i, \mathrm{id}]$ to $[A,\mathrm{id}]$ since $A$ contains $V_i$. Thus, $[A,\mathrm{id}]$ dominates $[R_i, \varphi_i]$. This concludes the proof of Claim~\ref{claim:Dominate}.

\medskip \noindent
Observe that Claim~\ref{claim:Dominate} implies that $\mathscr{C}$ is connected. Now, we want to prove that $\mathscr{C}$ is simply connected. Given an arbitrary combinatorial loop $\alpha$ in the one-skeleton of $\mathscr{C}$, we define its \emph{height} $\mathfrak{h}(\alpha)$ as the sum of the heights of its vertices and its \emph{complexity} by $\chi(\alpha):= ( \mathrm{lg}(\alpha),- \mathfrak{h}(\alpha))$ (ordered by lexicographic order){, where $\mathrm{lg}(\alpha)$ denotes the length of the path $\alpha$.}

\medskip \noindent
Let $\gamma$ be a combinatorial loop in $\mathscr{C}$ and $[A,\psi]$ be a vertex that dominates all the vertices in $\gamma$ (as given by Claim~\ref{claim:Dominate}). If $\gamma$ is not reduced to a single vertex, we define a new combinatorial loop $\gamma'$ as follows. Fix a vertex $x \in \gamma$ that has minimal height. Its neighbours $y$ and $z$ in $\gamma$ necessarily have higher heights, so it follows from Lemma~\ref{lem:GoingUp} that we can write $x$ as $[R,\varphi]$ and $y,z$ respectively as $[R \cup a, \varphi], [R \cup b, \varphi]$ where $a,b \in \lambda(R^C)$. If $y=z$, define $\gamma'$ from $\gamma$ by removing this backtrack; otherwise, define $\gamma'$ by replacing $x$ with $x':=[R \cup a \cup b, \varphi]$. Clearly, $\gamma'$ is homotopically equivalent to $\gamma$ and $\chi(\gamma')< \chi(\gamma)$. Also, observe that $x'$ is also dominated by $[A,\psi]$. Indeed, because $y$ is dominated by $[A,\psi]$, it follows from Lemma~\ref{lem:GoingUp} that $[A,\psi]$ can also be written as $[R \cup a \cup x_1 \cup \cdots \cup x_r, \varphi]$; similarly, it can also be written as $[R \cup b \cup y_1 \cup \cdots \cup y_s, \varphi]$. But:

\begin{fact}
For all subtrees $U,V \subset \mathcal{T}$ and element $\xi \in \mathcal{N}$, if $[U,\xi]=[V, \xi]$, then $U=V$.
\end{fact}

\noindent
So there must exist some index $1 \leq i \leq r$ such that $x_i=b$, which implies that there exists an increasing path from $y=[R \cup a,\varphi]$ to $[A,\psi]$ passing through $x'=[R \cup a \cup b, \varphi]$ (it suffices to add the $x_j$ in a well-chosen order). 

\medskip \noindent
Thus, we have constructed a new loop $\gamma'$ that is homotopically equivalent to $\gamma$, that has smaller complexity and all of whose vertices are dominated by $[A,\psi]$. By iterating the process, we get a sequence of combinatorial loops $\gamma,\gamma',\gamma'',\ldots$ that are pairwise homotopically equivalent. Observe that the height of each loop is bounded above by $\mathrm{lg}(\gamma) \mathfrak{h}([A,\psi])$, a constant that does not depend on the loop under consideration. Consequently, the process has to stop eventually, which is only possible if one of the loops is reduced to a single vertex. We conclude that $\gamma$ is homotopically trivial. 
\end{proof}

\paragraph{Descending links.} So far, we have only assumed that $\mathcal{T}$ is a locally finite rooted tree. From now on, and for the rest of the section, we assume that $\mathcal{T} = \mathcal{T}_{d,r}$ for some $d \geq 2$ and {$r \leq d$}.

\medskip
Given a vertex $x \in \mathscr{C}$, we distinguish the \emph{ascending} part $\mathrm{link}_\uparrow(x)$ of $\mathrm{link}(x)$ from its \emph{descending} part $\mathrm{link}_\downarrow(x)$. More precisely, $\mathrm{link}_\uparrow(x)$ (resp. $\mathrm{link}_\downarrow(x)$) denotes the subcomplex of $\mathrm{link}(x)$ generated by the neighbours of $x$ of higher (resp. lower) height. In this section, we focus on the descending part of the link. This is the most interesting part, since we will see later that the ascending link of a vertex is always a simplex. The following proposition is the main result of the subsection. It states that descending links of vertices belong to the following families of simplicial complexes. For all $p,q \geq 0$, $\mathcal{I}(p,q)$ denotes the simplicial complex whose vertices are the subsets of size $p$ in $\{1, \ldots, q\}$ and whose simplices are spanned by vertices given by pairwise disjoint subsets.

\medskip \noindent
In order to state our proposition, we need to introduction some notation. First, given a vertex $v$ in our rooted tree $\mathcal{T}$, we denote by $c(v)$ the set of all the children of $v$. Next, given an admissible subtree $R \subset \mathcal{T}$, a \emph{rigid permutation of the children of $R$} is a forest isomorphism $R^C \to R^C$ that preserves the left-right order on each connected component (thinking of $\mathcal{T}$ drawn on the plane so that the children of each vertex are naturally ordered from left to right). We denote by $\mathrm{Rig}(R)$ the (finite) subgroup of $\mathcal{N}$ given by the rigid permutations of the children of $R$. Observe that $\mathrm{Rig}(R)$ is isomorphic to the symmetric group $S_k$ where $k$ denotes the number of children of $R$.

\begin{proposition}\label{prop:DescendingLinks}
Let $x \in \mathscr{C}$ be a vertex. Assume that $x$ admits a representative of the form $(R,\mathrm{id})$, and fix a leaf $u$ of $R$. Then the map
$$\Phi : [R \backslash u, \sigma] \mapsto \sigma(c(u)), \ \sigma \in \mathrm{Rig}(R)$$ 
induces an isomorphism from $\mathrm{link}_{\downarrow}(x)$ to $\mathcal{I}(|c(u)|, \mathfrak{h}(x))$.
\end{proposition}

\begin{proof}
First of all, we need to justify that $\Phi$ is well-defined, i.e.\ every neighbour of $x$ of lower height can be written as $[R \backslash u,\sigma]$ for some $\sigma \in \mathrm{Rig}(R)$ and the choice of such a $\sigma$ does not modify the value of $\sigma(c(u))$. This is done by our first two claims.

\begin{claim}\label{claim:Choice}
For every neighbour $y$ of $x$ of lower height, there exists some $\sigma \in \mathrm{Rig}(R)$ such that $y=[R\backslash u,\sigma]$.
\end{claim}

\noindent
Let $(S,\varphi)$ be a representative of $y$. There must exist some $\ell\in \lambda(S^C)$ such that $x=[S \cup \ell, \varphi]$. The equality $[R,\mathrm{id}] = [S \cup \ell , \varphi]$ means that $\varphi$ induces a forest isomorphism $(S \cup \ell)^C \to R^C$ sending $\lambda((S \cup \ell)^C)$ to $\lambda(R^C)$. Since $\varphi(c(\ell))$ lies in $\lambda(R^C)$ and since $u$ is a leaf of $R$, there exists some $\sigma \in \mathrm{Rig}(R)$ such that $\sigma(\varphi(c(\ell)))= c(u)$. Then $\sigma \varphi$ extends to a forest isomorphism $S^C \to (R \backslash u)^C$ sending $\lambda(S^C)$ to $\lambda((R \backslash u)^C)$, which precisely means that $y=[S, \varphi]= [R \backslash u, \sigma^{-1}]$.

\begin{claim}\label{claim:Injectif}
For all $\mu,\nu \in \mathrm{Rig}(R)$, $[R \backslash u,\mu]= [R \backslash u, \nu]$ if and only if $\mu(c(u))=\nu(c(u))$. 
\end{claim}

\noindent
First, assume that $[R\backslash u, \mu]=[R \backslash u,\nu]$. We know that $\nu^{-1}\mu$ defines both a rigid permutation of the children of $R$ and a forest isomorphism $(R \backslash u)^C \to (R \backslash u)^C$ preserving $\lambda((R \backslash u)^C)$. Observe that $\nu^{-1}\mu$ has to fix $u$, since otherwise it would send some children of $u$ in $\lambda(R^C)$ to a children of a children of $R$, contradicting the fact that $\nu^{-1}\mu$ permutes the children of $R$. Hence $\nu(c(u))=\mu(c(u))$ as desired.

\medskip \noindent
Conversely, assume that $\nu(c(u))=\mu(c(u))$. Then the isomorphism $\nu^{-1} \mu : R^C \to R^C$ extends to $(R \backslash u)^C \to (R \backslash u)^C$ by sending $u$ to itself. Hence $[R \backslash u,\mu]= [R \backslash u, \nu]$. This concludes the proof of Claim~\ref{claim:Injectif}. 

\medskip \noindent
Observe that Claim~\ref{claim:Injectif} shows that $\Phi$ is well-defined and injective. Since it is clearly surjective, it follows that $\Phi$ defines a bijection from the vertices of $\mathrm{link}_{\downarrow}(x)$ to the vertices of $\mathcal{I}(|c(u)|, \mathfrak{h}(x))$. In order to conclude the proof of our proposition, it remains to show that $\Psi$ preserves the simplicial structure, which is a consequence of our next observation:

\begin{claim}\label{claim:SpanningCube}
For all $k \geq 2$ and $\sigma_1, \ldots, \sigma_k \in \mathrm{Rig}(R)$, the vertex $x$ and its neighbours $[R \backslash u, \sigma_1], \ldots, [R\backslash u, \sigma_k]$ span a $k$-cube if and only if $\sigma_1(c(u)), \ldots, \sigma_k(c(u))$ are pairwise disjoint.
\end{claim}

\noindent
First, assume that $x$ and its neighbours $[R \backslash u, \sigma_1], \ldots, [R\backslash u, \sigma_k]$ span a $k$-cube. By construction of $\mathscr{C}$, there exists an admissible subtree $A \subset \mathcal{T}$, $k$ pairwise distinct vertices $a_1, \ldots, a_k \in \lambda(A^C)$, and an element $\varphi \in \mathcal{N}$ such that
$$[A \cup a_1 \cup \cdots \cup a_{i-1} \cup a_{i+1} \cup \cdots \cup a_k, \varphi]= [R \backslash u, \sigma_i]$$
for every $1 \leq i \leq k$ and $[A \cup a_1 \cup \cdots \cup a_k, \varphi]= [R, \mathrm{id}]$. Given an index $1 \leq i \leq k$, $\sigma_i^{-1} \varphi$ defines an isomorphism
$$(A \cup a_1 \cup \cdots \cup a_{i-1} \cup a_{i+1} \cup \cdots \cup a_k)^C \to (R \backslash u)^C$$
sending $\lambda((A \cup a_1 \cup \cdots \cup a_{i-1} \cup a_{i+1} \cup \cdots \cup a_k)^C)$ to $\lambda((R \backslash u)^C)$. If $\sigma_i^{-1} \varphi(a_i)$ is distinct from $u$, then it has to be a children of $R$. Because $\sigma_i$ permutes the children of $R$, a fortiori $\varphi(a_i)$ is also a children of $R$, which implies that $\varphi$ sends the children of $a_i$ to children of children of $R$, contradicting the fact that $\varphi$ defines an isomorphism $ (A \cup a_1\cup \cdots \cup a_k)^C \to R^C$ sending $\lambda( (A \cup a_1\cup \cdots \cup a_k)^C)$ to $\lambda(R^C)$. So we must have $\sigma_i^{-1}\varphi(a_i)=u$, which implies that $\varphi(c(a_i))=\sigma_i(c(u))$. Consequently, the $\sigma_j(c(u))$ are pairwise disjoint if and only if so are the $c(a_j)$, which is clear since the $a_j$ are pairwise distinct.

\medskip \noindent
Conversely, assume that $\sigma_1(c(u)), \ldots, \sigma_k(c(u))$ are pairwise disjoint. As a consequence, $\lambda(R^C)$ must have cardinality $\geq kd$. On the other hand,
$$\begin{array}{lcl}|\lambda(R^C)| & \leq & |\partial_- R \backslash \{\text{root}\}| d + \# \{ \text{neighbours of the root not in $R$}\} \\ \\ & < & ( | \partial_- R \backslash \{ \text{root} \} | +1) d, \end{array}$$ {where $\partial_-A$, for a non-empty admissible subtree $A$, is the set of vertices in $A$ but adjacent to vertices not in $A$.} 
So $\partial_-R$ must contain at least $k$ vertices distinct from the root. Let $u_1, \ldots, u_k \in R$ be $k$ such vertices. Because $\sigma_1(c(u)), \ldots, \sigma_k(c(u))$ are pairwise disjoint, we can find a permutation $\varphi \in \mathrm{Rig}(R)$ of the children of $R$ such that $\varphi(c(u_i)) = \sigma_i(c(u))$ for every $1 \leq i \leq k$. Observe that, for every $1 \leq i \leq k$, $\sigma_i^{-1}\varphi$ defines an isomorphism $(R \backslash u_i)^C \to (R \backslash u)^C$ sending $\lambda((R \backslash u_i)^C)$ to $\lambda((R \backslash u)^C)$, so $[R \backslash u_i, \varphi]= [R \backslash u, \sigma_i]$. Also, $[R, \varphi]= [R, \mathrm{id}]$. Therefore, $x$ and its neighbours $[R \backslash u, \sigma_1], \ldots, [R\backslash u, \sigma_k]$ span the cube
$$ \left\{ \left[ S \cup \bigcup_{i \in I} u_i, \varphi \right] \mid  I \subset \{ 1, \ldots, k\} \right\}, \text{ where } S:= R \backslash \bigcup\limits_{i=1}^k u_i.$$
This concludes the proof of Claim~\ref{claim:SpanningCube}, and of our proposition.
\end{proof}

\paragraph{Proof of the CAT(0) property.}\label{section:Proof.} We are now ready to prove our main result {in the case of $\mathcal{T} = \mathcal{T}_{d,r}$ for some $d \geq 2$ and {$r \leq d$}.}

\begin{proof}[Proof of Theorem~\ref{thm:MainNeretin}.]
We already know from Proposition~\ref{prop:OneConnected} that $\mathscr{C}$ is simply connected. But it remains to show that links of vertices in $\mathscr{C}$ are simplicial flag complexes in order to conclude that $\mathscr{C}$ is CAT(0). First, observe that:

\begin{claim}\label{claim:Link}
For every vertex $x \in \mathscr{C}$, $\mathrm{link}(x)$ decomposes as the join of $\mathrm{link}_\downarrow(x)$ and $\mathrm{link}_\uparrow(x)$. 
\end{claim}

\noindent
Let $a_1, \ldots, a_r$ (resp. $b_1, \ldots, b_s$) be neighbours of $x$ defining a simplex in $\mathrm{link}_\uparrow(x)$ (resp. $\mathrm{link}_\downarrow(x)$). By definition of cubes in $\mathscr{C}$, there exist an admissible subtree $R \subset \mathcal{T}$, vertices $\ell_1, \ldots, \ell_s$ of $\lambda(R^C)$, and an element $\varphi \in \mathcal{N}$ such that
$$\left\{ \left[ R \cup \bigcup_{i \in I} \ell_i, \varphi \right] \mid I \subset \{1, \ldots, s\} \right\}$$
is the cube spanned by $x$ and its neighbours $b_1, \ldots, b_s$, namely $[R \cup \ell_1 \cup \cdots \cup \ell_s,\varphi]=x$ and 
$$[R \cup \ell_1 \cup \cdots \cup \ell_{i-1} \cup \ell_{i+1} \cup \cdots \cup \ell_s, \varphi]=b_i$$
for every $1 \leq i \leq s$. As a consequence of Lemma~\ref{lem:GoingUp}, we also know that, for every $1 \leq i \leq r$, there exist vertices $u_1, \ldots, u_r$ of $\lambda((R \cup \ell_1 \cup \cdots \cup \ell_s)^C)$ such that 
$$[R \cup \ell_1 \cup \cdots \cup \ell_s \cup u_i, \varphi]=a_i.$$
Then $x$ and its neighbours $a_1, \ldots, a_r, b_1, \ldots, b_s$ span the cube
$$\left\{ \left[ R \cup \bigcup_{i \in I} \ell_i \cup \bigcup_{j \in J} u_j, \varphi \right] \mid I \subset \{1, \ldots, s\}, J \subset \{1, \ldots, r\} \right\}$$
in $\mathscr{C}$. In other words, $a_1, \ldots, a_r, b_1, \ldots, b_s$ span a simplex in $\mathrm{link}(x)$. This concludes the proof of Claim~\ref{claim:Link}.

\medskip \noindent
Given a vertex $x \in \mathscr{C}$, one easily sees that $\mathrm{link}_\uparrow(x)$ is a simplex. Indeed, if $x_1, \ldots, x_r$ denote its neighbours of higher height, then fix a representative $(R, \varphi)$ of $x$ and let $u_1, \ldots, u_r$ be vertices of $\lambda(R^C)$ such that $x_i= [R \cup u_i, \varphi]$ for every $1 \leq i \leq r$ (as given by Lemma~\ref{lem:GoingUp}). Then
$$\left\{ \left[ R \cup \bigcup_{i \in I} u_i, \varphi \right] \mid I \subset \{1, \ldots, r\} \right\}$$
is a cube in $\mathscr{C}$ spanned by $x$ and its neighbours $x_1, \ldots, x_r$. This proves that $\mathrm{link}_\uparrow(x)$ is a simplex. Therefore, Claim~\ref{claim:Link} shows that $\mathrm{link}(x)$ decomposes as the join of a simplex together with a copy of $\mathcal{I}(d,\mathfrak{h}(x))$ and $\mathcal{I}(d,\mathfrak{h}(x))$ as given by Proposition~\ref{prop:DescendingLinks}. Such a complex being clearly simplicial and flag, we conclude that $\mathscr{C}$ is CAT(0).

\medskip \noindent
We know from Proposition~\ref{prop:LocallyFinite} that $\mathscr{C}$ is locally finite, so it remains to show that vertex-stabilisers in $\mathscr{C}$ are compact. This follows from our final observation:

\begin{claim}\label{claim:Stabilisers}
For all admissible subtree $R \subset \mathcal{T}$ and element $\varphi \in \mathcal{N}$, the stabiliser of the vertex $[R, \varphi]$ in $\mathcal{N}$ is $\varphi \cdot \mathrm{Aut}(R^C, \lambda(R^C)) \cdot \varphi^{-1}$, where $\mathrm{Aut}(R^C, \lambda(R^C))$ denotes the group of the forest automorphisms $R^C \to R^C$ stabilising $\lambda(R^C)$. 
\end{claim}

\noindent
If $\psi \in \mathcal{N}$ fixes $[R, \varphi]$, then $[R, \varphi]= \psi \cdot [R, \varphi]= [R, \psi \varphi]$, so $\varphi^{-1} \psi \varphi$ induces an automorphism $R^C \to R^C$ preserving $\lambda(R^C)$. In other words, $\psi$ belongs to $\varphi \cdot \mathrm{Aut}(R^C, \lambda(R^C)) \cdot \varphi^{-1}$. Conversely, it is clear that $\varphi \cdot \mathrm{Aut}(R^C,\lambda(R^C)) \cdot \varphi^{-1}$ fixes $[R,\varphi]$. 
\end{proof}

\begin{remark}\label{remark:CCinGeneral}
Our cube complex $\mathscr{C}(\mathcal{T})$ is defined for any locally finite rooted tree $\mathcal{T}$, but we only proved that it is CAT(0) when $\mathcal{T}$ is some $\mathcal{T}_{d,r}$ with $d\geq 2$ and $r \leq d$. Actually, the cube complex may not be CAT(0): for instance, by transferring \cite[Section~3.3]{GLU} in our setting, one can show that $\mathscr{C}(\mathcal{T})$ is not CAT(0) when $\mathcal{T}$ is the union of $n \geq 3$ infinite rays sharing a common origin. Nevertheless, by adapting the arguments from \cite[Theorem~3.3]{GLU}, one can show that $\mathscr{C}(\mathcal{T})$ is always contractible.
\end{remark}

\subsection{First applications}\label{section:FirstApplications}

In this section, we record a few direct consequences of Theorem~\ref{thm:MainNeretin}. Most of them are already available in the literature, but our approach allows us to prove them in a unified (and, sometimes, simpler) way. 

\begin{remark}
In our construction from Section~\ref{section:MainConstruction} and the applications below, we focused on full almost automorphism groups of regular trees. But our arguments apply almost word for word to specific subgroups of Neretin groups. For instance, fix a subgroup $D \leq \mathrm{Sym}(d)$, let $W(D) \leq \mathrm{Aut}(\T_d)$ denote the infinite iterated wreath product of $D$, and consider the subgroup $\mathrm{AAut}_D(\mathcal{T}_{d})$ of $\mathrm{AAut}(\mathcal{T}_{d})$ given by the forest isomorphisms $\mathcal{T}_{d} \backslash A \to \mathcal{T}_{d} \backslash B$ that restrict to elements of $W(D)$ on each connected component (once identified with a copy of $\mathcal{T}_d$). As particular cases, $\mathrm{AAut}_D(\mathcal{T}_{d})$ coincides with $\mathrm{AAut}(\mathcal{T}_{d})$ if $D= \mathrm{Sym}(d)$ and with Thompson's group $V_{d}$ if $D=\{1\}$. These groups have been introduced and studied in \cite{MR2806497}. Our arguments apply to all the groups in this family, whatever $D$ is.
\end{remark}

\paragraph{A-T-menabiliy.} A topological group is \emph{a-T-menable} if it admits a continuous and proper action by affine isometries on a Hilbert space. Because groups acting properly on CAT(0) cube complexes are automatically a-T-menable \cite{niblo1997groups}, it immediately follows from Theorem~\ref{thm:MainNeretin} that:

\begin{theorem}\label{thm:NeretinATmenable}
The Neretin groups {$\mathcal{N}_{d,r}$} are a-T-menable.
\end{theorem}

\noindent
Here, $\mathcal{N}_{d,r}$ refers to the almost automorphism group of the tree $\mathcal{T}_{d,r}$ previously defined, i.e.\ the rooted tree whose root has degree $r$ and all of whose other vertices have $d$ children. Although not explicitly stated, Theorem~\ref{thm:NeretinATmenable} can also be found in \cite[Section~3.3.3]{le2015geometrie}. There, proper \emph{commensurating actions} of Neretin groups are constructed. Formally, admitting such an action amounts to admitting a proper action on a CAT(0) cube complex \cite{cornulier_wallings}, so again we can conclude thanks to \cite{niblo1997groups} that Neretin groups are a-T-menable.

\paragraph{Subgroups with fixed-point properties.} It is proved in \cite{navas2002groupes} that, if a subgroup of a Neretin group satisfies Kazhdan's property (T), then it has to lie inside the automorphism group of a subforest. Theorem~\ref{thm:MainNeretin} (combined with Claim~\ref{claim:Stabilisers}) leads to the following improvement of this observation:

\begin{theorem}
Every subgroup of $\mathcal{N}_{d,r}$ satisfying the property $(\mathrm{FW}_{\mathrm{LocFin}})$ lies in the automorphism group of a {cofinite} rooted subforest.
\end{theorem}

\noindent
Here, we say that a group \emph{satisfies the property~$(\mathrm{FW}_{\mathrm{LocFin}})$} if all its actions on locally finite CAT(0) cube complexes admit global fixed points. Because Kazhdan's property~(T) can be characterised as a fixed-point property on (complete) median spaces \cite{medianviewpoint}, and since CAT(0) cube complexes can be thought of as median graphs \cite{mediangraphs, Roller}, the property~$(\mathrm{FW}_{\mathrm{LocFin}})$ can be thought of as a discrete and locally finite version of the property~(T).

\paragraph{Finiteness properties.} From \cite{MR3612334}, we know that Neretin groups are \emph{compactly presented}, i.e.\ they admit presentations with compact generating sets and relations of uniformly bounded length. More generally, it was proved in \cite{MR3801421} that they are \emph{of type $F_\infty$}, i.e.\ they act on contractible CW-complexes with compact-open stabilisers and with finitely many cell-orbits in each dimension. It turns out that our cube complexes also provide such CW-complexes, so we recover these two statements.

\begin{theorem}\label{thm:Connected}
Neretin groups $\mathcal{N}_{d,r}$ are of type $F_\infty$. In particular, they are compactly presented.
\end{theorem}

\noindent
The key point is to show that descending links in our cube complex become more and more connected as the height increases. As a consequence of Proposition~\ref{prop:DescendingLinks}, this assertion follows from the next lemma.

\begin{lemma}\label{lem:Connected}
Let $p,q \geq 1$ be two integers. For every $n \geq 0$, $\mathcal{I}(p,q)$ is $n$-connected as soon as $q \geq (3+2n)p$. 
\end{lemma}

\noindent
In the proof, for all $a,b,c \geq 1$ such that $b \leq c$, we identify $\mathcal{I}(a,b)$ with the subcomplex of $\mathcal{I}(a,c)$ spanned by the vertices given by the subsets in $\{1, \ldots, b\} \subset \{1, \ldots, c\}$.

\begin{proof}[Proof of Lemma~\ref{lem:Connected}.]
It is clear that $\mathcal{I}(p,q)$ is connected if $q \geq 3p$. (In fact, any two vertices admit a common neighbour.) Assume that our statement holds for some $n$. First, observe that: 
\begin{itemize}
	\item Given $a,b \geq 1$, $\mathcal{I}(a,b+1)$ is obtained from $\mathcal{I}(a,b)$ by gluing cones over copies of $\mathcal{I}(a,b+1-a)$. Indeed, no two vertices given by subsets of $\{1, \ldots, b+1\}$ containing $b+1$ are adjacent and the link of every such vertex is isomorphic to $\mathcal{I}(a,b+1-a)$. It follows that the inclusion $\mathcal{I}(a,b) \hookrightarrow \mathcal{I}(a,b+1)$ is $\pi_{n+1}$-surjective if $b \geq (4+2n)a$, since we know from our assumption that the latter inequality implies that $\mathcal{I}(a,b+1-a)$ is $n$-connected. 
	\item Given $a,b \geq 1$, $\mathcal{I}(a,b)$ is homotopically trivial in $\mathcal{I}(a,c)$ if $c \geq b+a$, since it lies in the link of the vertex $\{c,c-1, \ldots, c-a+1 \}$.
\end{itemize}
By combining these two observations, it follows that $\mathcal{I}(p,q)$ is $(n+1)$-connected if $q \geq (3+2(n+1))p$, as desired.
\end{proof}

\begin{proof}[Proof of Theorem~\ref{thm:Connected}.]
For every $n \geq 1$, let $\mathscr{C}(n)$ denote the subcomplex of $\mathscr{C}$ spanned by the vertices of height $\leq n$. We already know that $\mathcal{N}_{d,r}$ acts on each $\mathscr{C}(r)$ with compact-open stabilisers. Thanks to our next two claims, \cite[Corollary~4.11]{MR4019790} applies and shows that $\mathcal{N}_{d,r}$ is of type $F_\infty$.

\begin{claim}
For every $n \geq 1$, $\mathcal{N}_{d,r}$ acts on $\mathscr{C}(n)$ with only finitely many orbits of cells.
\end{claim}

\noindent
Let $Q$ be a cube in $\mathscr{C}(n)$. There exist an integer $m \geq 0$, an admissible subtree $R$, vertices $\ell_1, \ldots, \ell_m \in \lambda(R^C)$, and an element $\varphi \in \mathcal{N}_{d,r}$ such that
$$Q= \left\{ \left[ R \cup \bigcup\limits_{i \in I} \ell_i, \varphi \right] \mid I \subset \{1, \ldots, m\} \right\}.$$
Up to translating by $\varphi^{-1}$, we can assume that $\varphi= \mathrm{id}$. Because $R \cup \ell_1 \cup \cdots \cup \ell_m$ must lie in the first $n$ levels of $\mathcal{T}_{d,r}$, there exists a constant depending only on $n,d,r$ that bounds the number of possible choices for $m$, $R$, and $\ell_1, \ldots, \ell_m$. 

\begin{claim}
For all $k,n \geq 1$, $\mathscr{C}(n)$ is $k$-connected if $n\geq (3+2k)d-1$.
\end{claim}

\noindent
Because $\mathscr{C}$ is contractible (as a consequence of the CAT(0) property), it follows from Morse theory that it suffices to show that the descending link of a vertex of height $\geq (3+2k)d$ is $k$-connected. But this claims is a direct consequence of Proposition~\ref{prop:DescendingLinks} and Lemma~\ref{lem:Connected}.
\end{proof}

\subsection{A fixed-point theorem and its applications}\label{Subsection_fixed_point_theorem}

\noindent
In this section, our goal is to prove the following fixed-point theorem and to apply it to Neretin and Cremona groups. 

\begin{theorem}\label{thm:PurelyEll}
Let $G$ be a finitely generated group acting on a CAT(0) cube complex $X$ by elliptic isometries. If $X$ has no infinite cube, then $G$ has a global fixed point.
\end{theorem}

We postpone the proof of the theorem and first explain its applications.

\paragraph{Application to Neretin groups.} By combining Theorem~\ref{thm:MainNeretin} (and Claim~\ref{claim:Stabilisers}) with Theorem~\ref{thm:PurelyEll}, one immediately obtains:

\begin{theorem}\label{thm:purelyellipticneretin}
Let $H \leq \mathcal{N}_{d,r}$ be a finitely generated subgroup. If each element of $H$ induces an automorphism of some cofinite rooted subforest, then $H$ entirely lies in the automorphism group of a cofinite rooted subforest. 
\end{theorem}

\noindent
The statement also follows from \cite[Corollary~3.6]{le2018commensurated}. 

\begin{remark}
	As pointed out to us by the referee, Theorem~\ref{thm:purelyellipticneretin} can be generalised to arbitrary rooted trees of bounded degree. Indeed, if $T$ is a rooted tree of bounded degree, we can complete it to a regular rooted tree $T'$ by adding the missing vertices. The embedding of $T$ into $T'$ gives an embedding of $\AAut(T)$ into $\AAut(T')$. We can then apply Theorem~\ref{thm:purelyellipticneretin} to $\AAut(T')$. 
\end{remark}

\paragraph{Application to Cremona groups.} 
We can now prove our regularization theorem for finitely generated groups of birational transformations:

\begin{proof}[Proof of Theorem~\ref{thm:CremonaReg}]
	
	We can identify $\AAut(\FF\p^2(\F_q))$ with $\mathcal{N}_d$ by conjugating with an almost isomorphism between $\T_d$ and $\FF\p^2(\F_q)$. In that way we see $\Bir_{\F_q}(\p^2)$ as a subgroup of $\mathcal{N}_d$ for $d=q+1$, and we obtain an isometric action on the CAT(0) cube complex $\mathscr{C}$. 
		
	 Clearly, being elliptic is preserved by this identification. Note that a subgroup $G\subset \Bir_{\F_q}(\p^2)$ fixes a vertex in $\mathscr{C}$ if and only if there exists a finite subforest $F\subset \FF \p^2(\F_q)$  of the blow-up forest such that every element in $G$ can be represented by a triple of the form $(\psi, F, F)$. If we assume $G$ to be finitely generated, then Lemma~\ref{lem:CremonaElliptic} implies that $G$ is conjugate to a subgroup of $\BBir(S)$ for some regular projective surface $S$.
	
\medskip

	Let $L$ be the finite field extension of $\F_q$ such that all the base-points of the generators in $\Gamma$, and hence the base-points of all elements in $\Gamma$ are defined over $L$. We now consider the action of $\Bir(S_L)$ on the $\CAT(0)$ cube complex $\mathscr{C}$. By Lemma~\ref{lem:CremonaElliptic}, every element in $\Gamma$  is elliptic. Theorem~\ref{thm:PurelyEll} implies that $\Gamma$ fixes a vertex in $\mathscr{C}$.  By Lemma~\ref{lem:CremonaElliptic}, there exists a regular projective surface $S'_L$ defined over $L$ such that $\Gamma$ is conjugate by a birational map defined over $L$ to a subgroup of $\BBir(S'_L)$. Since all the base-points of $\Gamma$ are defined over $L$, this implies that $\Gamma$ is conjugate to a subgroup of $\Aut(S'_L)$. By \cite[Theorem~1.3]{lonjou-urech} there exists a regular projective surface $\tilde{S}$ defined over $\F_q$ such that $\Gamma$ is conjugate  (by a birational map defined over $\F_q$) to a subgroup of $\Aut(\tilde{S})$.
\end{proof}

Note that the above proof of Theorem~\ref{thm:CremonaReg} implicitely contains the following analogue of Theorem~\ref{thm:CremonaReg} for birational transformations that are bijective on $\F_q$-points.

\begin{proposition}\label{prop:BBirreg}
	Let $\kk$ be a finite field, $S$ a surface over $\kk$ and $\Gamma\subset\Bir(S)$ a finitely generated subgroup such that for every element in $\gamma\in\Gamma$ there exists a regular projective surface $S'$ over $\kk$ such that $\gamma$ is conjugate to an element in $\BBir(S')$, then there exists a regular projective surface $T$ over $\kk$ such that $\Gamma$ is conjugate to a subgroup of $\BBir(T)$.
\end{proposition}

\begin{proof}
	Again we consider the action of $\Gamma$ on the $\CAT(0)$ cube complex $\mathscr{C}$. Our assumption implies that every element in $\Gamma$ acts by an elliptic isometry on $\mathscr{C}$, so by Theorem~\ref{thm:PurelyEll}, $\Gamma$ fixes a vertex in $\mathscr{C}(S)$. Hence, by Lemma~\ref{lem:CremonaElliptic}, there exists a regular projective surface $T$ such that $\Gamma$ is conjugate  to a subgroup of $\BBir(T)$.
\end{proof}

\begin{remark}
	Let $S$ be a surface over a finite field $\kk$. By considering the action of $\Bir(S)$ by almost automorphisms on its rational blow-up forest $\FF S(\kk)$, we obtain an isometric action on the CAT(0) cube complex $\mathscr{C}$ associated to $\FF S(\kk)$. The cube complex $\mathscr{C}$ has the advantage of being locally compact. However, the subset of elements in $\Bir(S)$ fixing a vertex in $\mathscr{C}$ is very large: it consists exactly of the elements that are conjugate to a birational transformation that is bijective on $\kk$-points on some regular projective surface. 
	\medskip
	
	In \cite{lonjou-urech}, two other CAT(0) cube complexes are constructed with an isometric action of $\Bir(S)$. Most importantly, the \emph{blow-up complex}, whose vertex stabilisers correspond exactly to projectively regularisable elements in $\Bir(S)$. Therefore, the subset of elements in $\Bir(S)$ that induce loxodromic elements on the blow-up complex is strictly larger than the subset of elements in $\Bir(S)$ inducing loxodromic elements on $\mathscr{C}$. 
	\medskip
	
	Another weakening of the blow-up complex is the CAT(0) cube complex $\mathcal{C}^0$ constructed in \cite{lonjou-urech}. The elements in $\Bir(S)$ fixing a vertex in $\mathcal{C}^0$ are exactly the (not necessarily projectively) regularisable elements in $\Bir(S)$.  Hence, also the subset of elements in $\Bir(S)$ that induce loxodromic elements on the blow-up complex is strictly larger than the subset of elements in $\Bir(S)$ inducing loxodromic elements on $\mathcal{C}^0$. The main advantage of $\mathcal{C}^0$ is that it can be constructed for varieties of arbitrary dimension.  
\end{remark}

\paragraph{Proof of the fixed-point theorem.} From now on, we assume that the reader is familiar with the geometry of CAT(0) cube complexes. Before turning to Theorem \ref{thm:PurelyEll}, we need two preliminary lemmas. The first one is well-known and contained in the proof \cite[Theorem~5.1]{Sageev-ends_of_groups}. We include a proof for the reader's convenience.

\begin{lemma}\label{lem:ConvexHullOrbit}
	Let $G$ be a finitely generated group acting on a CAT(0) cube complex $X$. For every vertex $x_0 \in X$, the convex hull of the orbit $G \cdot x_0$ is a $G$-invariant convex subcomplex on which $G$ acts with finitely many orbits of hyperplanes.
\end{lemma}

\begin{proof}
	Fix a finite and symmetric generating set $S \subset G$. Let $Y$ denote the convex hull of $G \cdot x_0$, and let $J$ be a hyperplane crossing $Y$. There must exist $g_1,g_2 \in G$ such that $J$ separates $g_1x_0$ and $g_2x_0$. Write $g_1^{-1}g_2$ as a product $s_1 \cdots s_r$ where $s_1, \ldots, s_r \in S$. As a consequence of the convexity of halfspaces \cite{Sageev-ends_of_groups}, $J$ must separate $g_1s_1 \cdots s_{i-1} x_0$ and $g_1s_1 \cdots s_i x_0$ for some $1 \leq i \leq r$, which implies that $J$ admits a $G$-translate that separate $x_0$ and $s_ix_0$. Thus, we have proved that every hyperplane of $Y$ admits a $G$-translate in 
	$$\bigcup\limits_{s \in S} \left\{ \text{hyperplanes separating $x_0$ and $sx_0$} \right\},$$
which is finite since there exist only finitely many hyperplanes separating two given vertices \cite{Sageev-ends_of_groups}. This proves our lemma.
\end{proof}

\noindent
For our next lemma, we use the following notation. Given a CAT(0) cube complex $X$ and two vertices $x,y \in X$, we denote by $d_\infty(x,y)$ the maximal number of pairwise non-transverse hyperplanes that separate $x$ and $y$. Although this observation is not used in the sequel, $d_\infty$ turns out to be a metric and to coincide with the standard extension of the $\ell^\infty$-metrics defined on each cube of $X$ \cite[Corollary 2.5]{MR1111556}. {To a given hyperplane $J$ of a CAT(0) cube complex, we denote by $N(J)$ the \emph{carrier} of $J$, i.e.\ the subcomplex spanned by the edges crossed by $J$.}

\begin{lemma}\label{lem:RayDiamInfty}
Let $X$ be a CAT(0) cube complex and $\rho$ a geodesic ray. If $X$ does not contain an infinite cube, then $\rho$ has infinite diameter with respect to $d_\infty$.
\end{lemma}

\begin{proof}
Assume that $\rho$ has finite diameter with respect to $d_\infty$. For every $i \geq 0$, set
	$$\mathcal{H}_i:= \left\{ \text{hyperplanes $J$ crossing $\rho$ and satisfying $d_\infty(\rho(0),N(J))=i$} \right\}, $$
 Because a geodesic cannot cross a hyperplane twice \cite{Sageev-ends_of_groups}, there must exist infinitely many hyperplanes crossing $\rho$, so some of the $\mathcal{H}_i$ must be infinite. Let $m$ denote the smallest index such that $\mathcal{H}_m$ is infinite. Because $\mathcal{H}_1, \ldots, \mathcal{H}_{m-1}$ are all finite, we can fix a vertex $x \in \rho$ sufficiently far away so that all the hyperplanes in $\mathcal{H}_1,\ldots, \mathcal{H}_{m-1}$ separate $\rho(0)$ and $x$. Finally, let $\mathcal{H}$ denote the collection of the hyperplanes in $\mathcal{H}_m$ that do not separate $\rho(0)$ and $x$.
	
	\begin{fact}
		The hyperplanes in $\mathcal{H}$ are pairwise transverse.
	\end{fact}
	
	\noindent
	Assume for contradiction that $J_1,J_2 \in \mathcal{H}$ are distinct and non-transverse. Up to reindexing our hyperplanes, assume that $J_1$ separates $\rho(0)$ from $J_2$. Then 
	$$d_\infty(\rho(0),N(J_2)) \geq d_\infty(\rho(0), N(J_1)) +1 >m,$$
	contradicting the fact that $J_2$ belongs to $\mathcal{H}_m$.
	
	\begin{fact}
		For every $J \in \mathcal{H}$, $x \in N(J)$.
	\end{fact}
	
	\noindent
	Assume for contradiction that there exists some $J \in \mathcal{H}$ such that $x \notin N(J)$. As a consequence, there must exist a hyperplane $H$ separating $x$ from $N(J)$ \cite[Proposition~2.17]{MR2413337}. Because $x$ belongs to $\rho$ and that $J$ crosses $\rho$, this new hyperplane must also cross $\rho$. A fortiori, $H$ cannot cross $\rho$ between $\rho(0)$ and $x$, which implies that $H$ does not belong to $\mathcal{H}_1, \ldots, \mathcal{H}_{m-1}$, i.e.\ $d(\rho(0),N(H)) \geq k$. But $H$ separates $x$ from $J$, so
	$$d_\infty(\rho(0), N(J)) \geq d_\infty(\rho(0),N(H))+1 \geq m+1,$$
	contradicting the fact that $J$ belongs to $\mathcal{H}_m$.
	
	\medskip \noindent
	The combination of the two facts shows that the edges starting from $x$ and crossing the hyperplanes in $\mathcal{H}$ span an infinite cube.
\end{proof}

\begin{proof}[Proof of Theorem \ref{thm:PurelyEll}.]
	As a consequence of Lemma \ref{lem:ConvexHullOrbit}, we can assume without loss of generality that $X$ coincides with the convex hull of the orbit $G \cdot x_0$ for some $x_0 \in X$. We distinguish two cases, depending on whether $X$ is bounded or unbounded with respect to $d_\infty$.
	
	\medskip \noindent
	\textbf{Case 1:} $X$ is unbounded with respect to $d_\infty$.
	
	\medskip \noindent
Let $M$ denote the number of $G$-orbits of hyperplanes. We know that there exist two vertices $x,y \in X$ that are separated by more than $3M$ pairwise non-transverse hyperplanes. In this collection, we can find three hyperplanes that belong to the same $G$-orbit. In other words, there exist two elements $g,h \in G$ and one hyperplane $J$ such that $gJ$ separates $J$ and $hJ$. Let $J^+$ denote the halfspace delimited by $J$ that contains $hJ$. If $gJ^+ \subset J^+$, then $\{g^nJ^+ \mid n \geq 0\}$ defines a decreasing sequence of halfspaces, implying that $g$ has unbounded orbits. Similarly, if $hJ^+ \subset J^+$ then $h$ has unbounded orbits. Otherwise, if the inclusions $gJ^+ \subset J^+$ and $hJ^+ \subset J^+$ do not hold, then we must have $gh^{-1} \cdot hJ^+ \subset hJ^+$, proving that $gh^{-1}$ has unbounded orbits. Because $G$ only contains elements with bounded orbits, we conclude that this first case cannot happen.
	
	\medskip \noindent
	\textbf{Case 2:} $X$ is bounded with respect to $d_\infty$.
	
	\medskip \noindent
	As a consequence, given a halfspace $D$ delimited by a hyperplane $J$, we can define the \emph{depth} of $D$ by 
	$$p(D):= \max \left\{ d_\infty(x,N(J)) \mid x \in D \right\}.$$
	A hyperplane is \emph{balanced} if the two halfspaces it delimits have the same depth, and it is \emph{unbalanced} otherwise. If $J$ is an unbalanced hyperplane, we denote by $J^+$ the hyperplane it delimits that has the larger depth; and if it is balanced, we denote by $p(J)$ the common depth of the two halfspaces it delimits.
	
	\begin{fact}
		If $J_1,J_2$ are two unbalanced hyperplanes, then $J_1^+ \cap J_2^+ \neq \emptyset$.
	\end{fact}
	
	\noindent
	Assume for contradiction that $J_1^+ \cap J_2^+ = \emptyset$. Up to reindexing our hyperplanes, assume that $p(J_1^+) \geq p(J_2^+)$ and fix a vertex $x \in J_1^+$ satisfying $d_\infty(x,N(J_1)) = p(J_1^+)$. Because $J_1$ separates $x$ from $J_2$, we must have
	$$p(J_2^-) \geq d_\infty(x,N(J_2)) > d_\infty(x,N(J_1)) = p(J_1^+) \geq p(J_2^+)$$
	where $J_2^-$ denotes the complement of $J_2^+$. Thus, we obtain a contradiction with the definition of $J_2^+$.
	
	\begin{fact}\label{fact:BalancedTransverse}
		Any two balanced hyperplanes are transverse.
	\end{fact}
	
	\noindent
	Assume for contradiction that there exist two balanced hyperplanes $J_1,J_2$ that are not transverse. Up to reindexing our hyperplanes, assume that $p(J_1) \geq p(J_2)$. Fix a vertex $x$ in the halfspace delimited by $J_1$ disjoint from $J_2$ that satisfies $d_\infty(x,N(J_1)) = p(J_1)$. Because $J_1$ separates $x$ from $J_2$, we must have
	$$p(J_2) \geq d_\infty(x,N(J_2)) > d_\infty(x,N(J_1)) = p(J_1) \geq p(J_2),$$
	a contradiction.
	
	\medskip \noindent
	If the intersection
	$$\bigcap\limits_{\text{$J$ unbalanced}} J^+$$
	is non-empty, then it defines a $G$-invariant convex subcomplex as a consequence of the convexity of halfspaces \cite{Sageev-ends_of_groups}. Moreover, according to Fact \ref{fact:BalancedTransverse}, its hyperplanes are pairwise transverse, so (as a consequence of \cite[Proposition~2.1]{sageev_lecturenotes} for instance) it must be a cube, possibly infinite-dimensional. But we know by assumption that $X$ does not contain an infinite-dimensional cube, so we conclude that $G$ stabilises a finite cube. A fortiori, its orbits are bounded, as desired.
	
	\medskip \noindent
	Next, we assume that the previous intersection is empty and we want to find a contradiction. This will conclude the proof of our theorem. We construct a sequence of vertices $(x_i)_{i \geq 0}$ and a sequence of unbalanced hyperplanes $(J_i)_{i\geq 1}$ in the following way.
	\begin{itemize}
		\item We fix an arbitrary vertex $x_0 \in X$.
		\item If $x_0,\ldots, x_i$ and $J_1,\ldots, J_i$ are defined, we fix a, unbalanced hyperplane $J_{i+1}$ satisfying $x_i \notin J_{i+1}^+$ and we define the vertex $x_{i+1}$ as the projection of $x_i$ onto $J_1^+ \cap \cdots \cap J_{i+1}^+$. (Here, the projection refers to the nearest-point projection; it is well-defined according to \cite[Lemma~13.8]{MR2377497} as soon as our intersection is non-empty, which follows from the Helly property on convex subcomplexes \cite[Corollary~2.22]{MR2413337}.)
	\end{itemize}
	Observe that, by construction, we have $x_i \in \bigcap\limits_{k=1}^i J_k^+ \backslash J_{i+1}^+$ for every $i \geq 0$. For every $i \geq 0$, fix a geodesic $[x_i,x_{i+1}]$ between $x_i,x_{i+1}$; and notice that the concatenation $[x_0,x_1] \cup [x_1, x_2] \cup \cdots$ defines a geodesic ray. Indeed, let $J$ be hyperplane crossing this infinite path. Let $k$ denote the smallest index such that $J$ crosses $[x_k,x_{k+1}]$. Because $x_{k+1}$ is defined as the projection of $x_k$ onto $J_1^+ \cap \cdots \cap J_{k+1}^+$, $J$ must be disjoint from this intersection \cite[Lemma~13.8]{MR2377497}. But we know that $[x_{k+1},x_{k+2}] \cup [x_{k+2},x_{k+3}] \cup \cdots$ lies in this intersection. Consequently, $J$ cannot cross the path $[x_0,x_1] \cup [x_1,x_2] \cup \cdots$ twice, proving that it defines a geodesic \cite{Sageev-ends_of_groups}. As a consequence of Lemma \ref{lem:RayDiamInfty}, $X$ must have infinite diameter with respect to $d_\infty$, which contradicts our assumption.
\end{proof}

\addcontentsline{toc}{section}{References}
\bibliographystyle{amsalpha}
{\footnotesize\bibliography{bibliography_cu}}
\Address
\end{document}